\documentclass[reqno,11pt,a4paper]{amsart}

\usepackage[hmargin=3.5cm,foot=0.7cm]{geometry}
\usepackage[utf8x]{inputenc}

\usepackage[english]{babel}

\usepackage[initials]{amsrefs}
\usepackage[all]{xy}
\usepackage{tikz}
\usetikzlibrary{matrix,arrows}

\usepackage{xcolor}
\usepackage{amssymb,amsmath}
\usepackage{enumerate}

\BibSpec{article}{%
  +{}{\PrintAuthors}  		{author}
  +{,}{ \textit}     		{title}
  +{,}{ }             		{journal}
  +{}{ \textbf}       		{volume}
  +{}{ \parenthesize} 		{date}
  +{,}{ }      	      		{conference}
  +{,}{ }      	      		{book}
  +{,}{ }            		{pages}
  +{,}{ }            	 	{note}
  +{,}{ }            	 	{status}
  +{,}{  \texttt } {eprint}
  +{.}{}              {transition}
  +{}  {\SentenceSpace \PrintReviews} {review}
}
\BibSpec{book}{%
  +{}{\PrintAuthors}  {author}
  +{,}{ \textit}      {title}
  +{,}{ }	      {note}
  +{,}{ }             {publisher}
  +{,}{ }             {place}
  +{,}{ }             {date}
  +{.}{}              {transition}
  +{}  {\SentenceSpace \PrintReviews} {review}
}

\newtheorem{Proposition}{Proposition}[section]
\newtheorem{Definition}[Proposition]{Definition}
\newtheorem{Lemma}[Proposition]{Lemma}
\newtheorem{Theorem}[Proposition]{Theorem}
\newtheorem{Corollary}[Proposition]{Corollary}

\newtheorem{Remark}[Proposition]{Remark}
\DeclareMathOperator{\Val}{Val}
\DeclareMathOperator{\Curv}{Curv}

\DeclareMathOperator{\vol}{vol}

\DeclareMathOperator{\Kl}{Kl}
\DeclareMathOperator{\Gr}{Gr}

\DeclareMathOperator{\im}{Im}
\DeclareMathOperator{\re}{Re}

\DeclareMathOperator{\glob}{glob}

\DeclareMathOperator{\Sp}{Sp}
\DeclareMathOperator{\SO}{SO}
\DeclareMathOperator{\GL}{GL}
\DeclareMathOperator{\SU}{SU}
\DeclareMathOperator{\G}{G}

 \renewcommand{\S}{ S_\lambda}
 \renewcommand{\O}{\mathbb O}

\newcommand{\spann}{\mathrm{span}}
\newcommand{\R}{\mathbb{R}}
\newcommand{\C}{\mathbb{C}}

\newcommand{\calC}{\mathcal{C}}
\newcommand{\calK}{\mathcal{K}}
\newcommand{\calV}{\mathcal{V}}
\newcommand{\calP}{\mathcal{P}}
\newcommand \lcur{[\![}
\newcommand \rcur{]\!]}
\newcommand{\pd}{\mathrm{pd}}
\newcommand{\Spin}{\mathrm{Spin}}

\newcommand{\AGr}{\overline{\Gr}}
\newcommand{\largewedge}{\mbox{\Large $\wedge$}}

\newcommand{\tf}{\omega} 
\newcommand{\cf}{\varphi} 
\newcommand{\Cf}{\Phi} 
\newcommand{\kfdot}{\odot}

\title[]{Integral geometry of exceptional spheres} 

\author{Gil Solanes}
\author{Thomas Wannerer}
 
\email{solanes@mat.uab.cat}
\email{thomas.wannerer@uni-jena.de}

\address{Departament de Matem\`atiques, Universitat Aut\`onoma de Barcelona, 08193 Bellaterra, Spain}

\address{Fakult\"at f\"ur Mathematik und Informatik, Friedrich-Schiller-Universit\"at Jena, 07743 Jena, Germany}

\thanks{Gil Solanes is a Serra H\'unter fellow and was supported by FEDER/MINECO grants IEDI-2015-00634 and MTM2015-66165-P. Thomas Wannerer was supported by DFG grant WA 3510/1-1.}

\begin{document}

\begin{abstract}
The algebras of valuations on $S^6$ and $S^7$ invariant under the actions of $\mathrm G_2$ and $\Spin(7)$ are shown to be isomorphic 
to the algebra of translation-invariant valuations on the tangent space at a point invariant under the action of the isotropy group. This is in analogy 
with the cases of real and complex space forms, suggesting the possibility that the same phenomenon holds in all Riemannian isotropic spaces.
Based on the description of the algebras the full array of kinematic formulas for invariant valuations and curvature measures in $S^6$ and $S^7$ is computed.
A key technical point is an extension of the classical theorems of Klain and Schneider on
simple valuations.  
\end{abstract}

\maketitle

\section{Introduction}
Exploiting various algebraic structures introduced by S.~Alesker \cites{alesker_irred,alesker_product,alesker_hard_lefschetz} on the space of valuations,
a conceptual breakthrough in integral geometry was recently achieved by A.~Bernig and J.H.G.~Fu.
At the  core of their new approach lies the observation that the kinematic formulas encode a structure
which is equivalent to the product of valuations discovered by Alesker. Employing more of the algebraic machinery developed by
Alesker, Bernig and Fu  \cite{hig} succeeded in computing the 
principal kinematic formula for
the action of the affine unitary group on $\C^n$ and described a simple procedure to obtain the full array of kinematic formulas.

Since then, this new approach to integral geometry has developed quite rapidly. The reader might consult 
\cite{fu_bcn} or the recent results \cites{bernig_voide,bernig_hug,bernig_solanes,wannerer_module,wannerer_area} for what is currently known about kinematic formulas in flat isotropic spaces,  
\cites{fu_pokorny_rataj,prz} for the state of the art regarding the regularity needed for sets to satisfy kinematic formulas,
\cite{abardia_wannerer,abs, haberl_schuster,schuster_crofton} for isoperimetric-type inequalities arising in integral geometry, or
\cites{bernig_faifman,faifman,ludwig_reitzner} for integral geometry  under large non-compact groups.

More recently, Alesker has developed a theory of valuations
on manifolds generalizing the classical theory of valuations on convex bodies in affine spaces \cites{valmfdsI,valmfdsII,valmfdsIII,valmfdsIV,valmfdsIG}. In particular,  a natural algebra structure  has been found on the space of valuations on any smooth manifold.
In combination with previous results by Fu \cite{fu_indiana}, this theory provides
 a framework for investigating the integral geometry of any  Riemannian isotropic space 
(i.e.\ a Riemannian manifold $M$ with a Lie group of isometries $G$ acting transitively on the sphere bundle $SM$). These pairs $(M,G)$ with $G$ connected  were classified by Tits \cite{tits} and Wang \cite{wang} as follows.
If $M$ is simply connected, then $M=V$ is flat euclidean space under the action of the groups $H\ltimes V$ with $H$ from the list
$$\SO(n), \ \mathrm U(n),\ \SU(n),\ \Sp(n),\ \Sp(n)\mathrm U(1),\ \Sp(n)\Sp(1),$$
$$\mathrm G_2,\ \Spin(7),\ \Spin(9);$$
or  $M$ is a real, complex, quaternionic, or octonionic space form under the action of the identity component of the isometry group; or $M$ 
is $S^6$ or $S^7$ under the action of $\mathrm G_2$ or $\Spin(7)$ respectively. 
If $M$ is not simply connected, then $M$ is real projective space under action of the identity component of the isometry group or  $\R P^6$ or $\R P^7$ under the action of $\mathrm G_2$ or $\Spin(7)/\{\pm1\}$ respectively.

In all these spaces kinematic formulas exist in the following form. Let $\calV(M)^G$ denote the space of smooth $G$-invariant valuations on $M$ (cf. Definition \ref{def_val}). There exists a linear map $k\colon \calV(M)^G\to \calV(M)^G\otimes \calV(M)^G$ which encodes the kinematic formulas in the sense that
\[
 k(\varphi)(A,B)=\int_G \varphi(A\cap gB)\; dg,\qquad  \varphi\in\calV(M)^G.
\]
Here $dg$ is a suitably normalized Haar measure on $G$ and $A,B\subset M$ are sufficiently nice compact sets. There are also kinematic formulas at the local level:
for the space $\calC(M)^G$ of $G$-invariant smooth curvature measures (cf. Definition \ref{def_curv}) there is a linear map $K\colon \calC(M)^G\to \calC(M)^G\otimes \calC(M)^G$ such that
\[
 K(\Psi)(A,U,B,V)=\int_G \Psi(A\cap gB,U\cap gV)\;dg,\qquad  \Psi\in\calC(M)^G
\]
where $A,B$ are as before and $U,V\subset M$ are Borel sets.

Building on \cites{hig,ags} and the foundational work of Alesker, the global kinematic operator $k$ was computed in complex projective  $\C P^n$ and complex hyperbolic space $\C H^n$ by Bernig, Fu, and the first-named author in \cite{bfs}. This allowed also the computation of the local kinematic operator $K$ on $\C^n$, and on $\C P^n, \C H^n$ as well. 
In fact, by Howard's transfer principle \cite{howard,bfs} there is a natural identification between the invariant curvature measures of any pair of isotropic spaces with the same isotropy group, and the local kinematic operators correspond to each other under this identification.

Compared to the classical case of real space forms, the complex case turned out to be much richer,
displaying several unexpected phenomena  for which a geometric explanation is still missing. 
Among them, the most striking fact is that the algebras of invariant valuations on $\C P^n, \C H^n$ and $\C^n$ are isomorphic. 
For the real space forms, the same is true for trivial reasons, but in the complex case no canonical construction explaining the existence of an isomorphism is known.

It is then natural to ask the following question:

\bigskip

\textit{Given two isotropic spaces  with isomorphic isotropy groups, are their algebras
of invariant valuations isomorphic?}

\bigskip

In the present paper we answer this question in the affirmative for the exceptional spheres $S^6$ and $S^7$ under the respective actions of $\G_2$ and $\Spin(7)$. 

\begin{Theorem} \label{Theorem_isomorphism} There exist isomorphisms between the algebras of invariant valuations on 
$$(S^6, \mathrm G_2)\qquad \text{and}\qquad(\C^3, \overline{\SU(3)})$$
and also between the algebras of invariant valuations on
$$(S^7,\Spin(7)) \qquad \text{and}\qquad (\R^7,  \overline{\mathrm G_2})$$
where $\overline{H}=H\ltimes V$ stands for the group generated by $H$ and the translations of the vector space $V$ on which $H$ acts linearly.
\end{Theorem}

In the flat spaces $(\C^n,\overline{\SU(n)})$  and $(\R^7,\overline{\G_2})$ the algebra structure of the invariant valuations has
been described by Bernig in \cites{bernig_sun, bernig_g2},
together with the global kinematic operator $k$.
Theorem \ref{Theorem_isomorphism} allows us to compute  explicitly not only  this operator in the exceptional spheres  
but  also  the local kinematic formulas which were not previously known (Theorems~\ref{Theorem_local_kinematic_SU3} and \ref{Theorem_local_kinematic_formulas_G2}).
We illustrate these results by bounding the mean intersection number of Lagrangian and associative submanifolds with arbitrary submanifolds of complementary
dimension (Corollaries~\ref{cor application G2} and \ref{cor application Spin}).

It turns out that the relation between curvature measures and valuations in $S^6$ and $S^7$ is  more subtle than in complex space forms. 
For instance, while every invariant curvature measure globalizing to zero in $\C^n$ has non-trivial globalization in $\C P^n$,  there are invariant curvature measures
that lie in the kernel of  the globalizations to $\C^3$ and to $S^6$. The same happens for $S^7$ and $\R^7$. It is also remarkable that in each of these geometries one encounters invariant
curvature measures which are odd. This is interesting, since all invariant valuations in isotropic affine spaces have been found to be even, but a geometric explanation of this fact is missing.

A key ingredient in our work is the fact that the algebra of valuations on a smooth manifold has an associated graded algebra
which is isomorphic to the space of sections of a certain bundle. 
More precisely, there is a natural filtration $\calV(M)=\calV_0(M)\supset\cdots \supset\calV_n(M)$ which is compatible with the product of valuations. 
Alesker \cite{valmfdsI}
constructed an isomorphism between the induced graded algebra $\bigoplus_i \calV_i/\calV_{i+1}$ and the space of smooth sections of the bundle $\Val(TM)$ of
translation-invariant valuations on the tangent spaces of $M$. In Section \ref{transfer} we define a similar filtration $\calC_0(M)\supset\cdots\supset\calC_n(M)$ 
on the space of curvature measures. This space has a natural module structure over the algebra of valuations which turns $\bigoplus_i\calC_i/\calC_{i+1}$ into a graded module over the graded algebra. 

\begin{Theorem} \label{Theorem_graded_modules}
 There exists a canonical isomorphism of graded modules 
 $$\bigoplus_{i=0}^{n} \calC_i/\calC_{i+1} \cong \Gamma(\Curv(TM)),$$
 where $\Curv(TM)$ is the bundle of translation-invariant smooth curvature measures  on the tangent spaces of $M$. 
\end{Theorem}

Another key point  in our study of the integral geometry of $S^7$ is the  restriction of invariant valuations and curvature measures to $S^6$, viewed as a totally geodesic subsphere. 
This leads naturally  to the problem of describing simple valuations and  curvature measures. 
Recall that a  valuation $\psi$ on an $n$-dimensional vector space $V$  is called simple if  
$$ \psi(K) =0$$
for every convex body $K$ contained in an affine  hyperplane $E\subset V$.
 For translation-invariant and continuous valuations on $V$, a complete
 characterization of simple valuations is  provided by the classical theorems of Klain \cite{klain_short} and Schneider 
\cite{schneider}: for every such valuation the homogeneous components of degrees smaller than $n-1$ must vanish, and the degree $n-1$ component must be odd. 
In Section \ref{simple} we extend this characterization to simple valuations  that are smooth, but not necessarily translation-invariant. 
On a smooth manifold it does not make sense to speak of even and odd valuations, but one can decompose 
$\calV_i(M)= \calV_i^+(M)\oplus  \calV_i^-(M)$ into the eigenspaces of the Euler-Verdier involution $\sigma$ (see Section~\ref{Sect_Background}).

\begin{Theorem}\label{Theorem_simple_valuations}
For every smooth valuation $\psi\in \calV(V)$ on an $n$-dimensional vector space $V$ the following are equivalent:
\begin{enumerate}[i)]
 \item $\psi$ is simple
 \item $\psi \in \calV_{n-1}$ and $\sigma(\psi)=(-1)^n \psi$
 \item $\iota^*\psi=0$ for every smoothly embedded manifold $\iota\colon N\hookrightarrow V$ with $\dim N<n$.
\end{enumerate}
\end{Theorem}

As a corollary, we obtain that the restriction of invariant curvature measures from $S^7$ to $S^6$ is essentially injective.

\medskip\noindent\textbf{Acknowledgments.} We thank Semyon Alesker and Andreas Bernig for useful comments on a previous version of this paper.

\section{Background}

\subsection{Octonionic linear algebra} Let $\O$ denote the normed algebra of the octonions with 
euclidean inner product 
$$\langle x,y\rangle = \re \overline{x}y =\re x\overline{y},$$
where $\overline{x}$ denotes conjugation in $\O$. 
Readers not familiar with the octonions may wish to consult  Baez's excellent  survey article  \cite{baez} or Chapter~6 of \cite{harvey}.

Let $\im\O=1^\bot$ be the hyperplane of pure octonions. The associative 3-form $\phi\in \largewedge^3(\im \O)^*$ is given by

$$\phi(u,v,w)=\langle u,v\cdot w\rangle,\qquad u,v,w\in\im\O.$$
With respect to a suitable orthonormal basis $1=e_0, e_1,\ldots, e_7$ of $\O$, we have   
\begin{equation}\label{eq associative 3-form}
 \phi= e^{123} + e^{145} + e^{167} + e^{246} - e^{257} - e^{347} - e^{356},
\end{equation}
where we write $e^{ijk}$ for $e^i \wedge e^j \wedge e^k$, denoting by $e^0,\ldots, e^7$ the dual basis to $e_0,\ldots, e_7$. 

Observe that $e_i\cdot e_j= \phi(e_i,e_j,e_k) e_k$  (here and throughout the text we use the convention of summation over repeated indices). Thus the octonionic product is completely determined by the above expression of $\phi$. 
In particular, the following identities hold
\[
 e_1\cdot e_2=e_3,\qquad e_1\cdot e_4=e_5,\qquad e_1\cdot e_6=e_7,\qquad e_2\cdot e_4=e_6,
 \]\[
 e_2\cdot e_5=-e_7,\qquad e_3\cdot e_4=-e_7,\qquad e_3\cdot e_5=-e_6. 
 \] 
Moreover, $e_i\cdot e_j=-e_j\cdot e_i$ if $i\neq j$, and  $e_j\cdot e_k=e_i$ if $e_{i}\cdot e_j=e_k$ with $1\leq i,j,k\leq 7$. 

The Hodge dual $\psi=*\phi\in\largewedge^4( \im\O)^*$, is given by
$$\psi = e^{4567} + e^{2367} + e^{2345} + e^{1357} -  e^{1346} - e^{1256} - e^{1247}$$
where $e^{ijkl}=e^i\wedge e^j\wedge e^k\wedge e^l$. 

The triple cross-product in  $\O$ 
$$ y\times z \times w = \frac 1 2 \left( (y\overline z) w - (w \overline z)y\right) , \qquad x,y,z,w\in \O,$$
is alternating and thus defines a $4$-form 
$$\Phi(x,y,z,w) = \langle x, y\times z \times w\rangle, \qquad x,y,z,w\in \O$$
called the Cayley calibration.
An alternate expression  (cf.\ \cite{salamon_walpuski}*{Theorem 5.20}) is
\begin{equation}\label{eq Cayley associative}\Phi = e^0 \wedge \phi + \psi.\end{equation}
In particular, the Cayley calibration is self-dual, $*\Phi= \Phi$. 

On $\O=\C^4$ with complex structure given by left-multiplication by $e_1$ the standard K\"ahler form and the complex determinant  are
$$\omega= e^0\wedge e^1 + e^2 \wedge e^3 + e^4 \wedge e^5 + e^6 \wedge e^7\in \largewedge^2\O^*$$
and 
$$\Omega=  (e^0 + i e^1) \wedge  (e^2 + i e^3) \wedge  (e^4 + i e^5)\wedge  (e^6 + i e^7)\in \largewedge^4\O^*\otimes \C.$$
The Cayley calibration may be expressed as
 \begin{equation}\label{eq Cayley Kahler}\Phi= \frac 1 2 \omega \wedge \omega + \re  \Omega\end{equation}

The exceptional Lie group $\G_2$ may be defined as  the automorphism group of the octonions
$$ \G_2=\{ g\in \GL(\O)\colon g(xy)=g(x)g(y)\ \text{for all}\ x,y\in \O\},$$
where $\GL(\O)$ denotes the group of invertible $\R$-linear endomorphisms of $\O$. As usual $\Spin(7)$ denotes the double cover of $\SO(7)$. 
These groups  can alternatively be  described as those linear transformations fixing the associative form $\phi$ or the Cayley calibration $\Phi$, respectively.

\begin{Proposition}[\cite{harvey}*{Theorem~6.80} and \cite{salamon_walpuski}*{Lemma 9.2}] \label{Prop G2 Spin7}
 \begin{align*} \G_2 & =\{ g\in \GL(\im \O)\colon g^* \phi = \phi\}\\
   \Spin(7) &= \{ g\in \GL(\O)\colon g^* \Phi = \Phi\}\end{align*}
\end{Proposition}

The preceding proposition together with the identities \eqref{eq associative 3-form},  \eqref{eq Cayley associative}, and \eqref{eq Cayley Kahler} implies the inclusions
\begin{equation}\label{eq inclusions G2 Spin7} \SU(3)\subset \G_2\subset  \Spin(7)\qquad \text{and}\qquad \SU(4)\subset \Spin(7)\end{equation}
and the following 

\begin{Proposition}[\cite{salamon_walpuski}*{Theorems 8.1 and 9.1}]\label{prop_transitive}
$\G_2$ acts transitively on $S^6\subset \im \O$ with isotropy group $\SU(3)$, and $\Spin(7)$ acts transitively on $S^7\subset \O$ with isotropy group $\mathrm G_2$. 
\end{Proposition}

\subsection{Valuations and curvature measures}
\label{Sect_Background}
\subsubsection{Valuations on vector spaces}
Let $V$ be a finite dimensional real vector space, and let $\mathcal K(V)$ denote the space of compact convex sets in $V$. 
A {\em valuation} on $V$ is a functional $\varphi\colon\mathcal K(V)\to\C$ such that 
\[
\varphi(A\cup B)=\varphi(A)+\varphi(B)-\varphi(A\cap B)
\]
whenever $A,B,A\cup B\in\mathcal K(V)$. 
The space of continuous (with respect to Hausdorff topology) translation-invariant valuations is denoted $\Val=\Val(V)$. The McMullen decomposition  of this space is (cf.\ \cite{mcmullen})
\begin{equation}\label{eq_mcmullen}
 \Val=\bigoplus_{k,\epsilon}\Val_k^\epsilon
\end{equation}
where the sum extends over $ \epsilon=\pm$, $k=0,\ldots,\dim V$ and $\Val_k^+$ (resp. $\Val_k^-$) denotes the subspace of even (resp. odd) $k$-homogeneous valuations: 
$\varphi\in \Val_k^\epsilon$ iff $\varphi(\lambda A)=\lambda^k\varphi(A)$ and $\varphi(-A)=\epsilon\varphi(A)$ for every $\lambda>0$ and $A\in\mathcal K(V)$.

Alesker introduced the  dense subspace $\Val^\infty\subset \Val$ of {\em smooth} valuations. 
One of the main features of this subspace is that it carries a number of  algebraic structures. Another salient feature is that a 
large part of these constructions extend to general smooth manifolds. Surprisingly, there exists also a rich theory of valuations without any continuity assumptions, see, e.g.,
\cites{ludwig_silverstein,ludwig_boroczky,mcmullen_lattice}.
\subsubsection{Klain functions}
A  homogeneous even  continuous translation-invariant valuation admits  a simple description in terms of its Klain function \cite{klain_even}. Given $\varphi\in\Val_k^+$, this is a function $\Kl_\varphi$ defined on the Grassmannian of $k$-planes $\Gr_k(V)$ by
\[
 \varphi(A)=\Kl_\varphi(E)\vol_k(A),\qquad E\in\Gr_k(V), A\in\mathcal K(E).
\]
Here  for simplicity we endow $V$ with an euclidean structure, and $\vol_k$ denotes the Lebesgue measure on $E\subset V$ with the induced euclidean structure. It was shown by Klain that $\varphi = \psi$ if and only if $\Kl_\varphi=\Kl_\psi$.

\subsubsection{Valuations on manifolds}

Recently, a theory of valuations on manifolds has been developed by Alesker \cites{valmfdsI,valmfdsII,valmfdsIII,valmfdsIV,valmfdsIG}. In the following, let $M$ be a smooth oriented manifold of dimension $n$
and let $\pi\colon SM\to M$ denote the cosphere bundle of $M$, that is the bundle over $M$ whose fiber over $x$ is $S_xM=(T^*_xM\setminus\{0\})/\R^+$.

Let  $\calP(M)$ denote  the space of compact submanifolds with corners (or smooth simple polyhedra) on $M$.
For $A\in\calP(M)$, the normal cycle $N(A)$ is an oriented Lipschitz submanifold of the cosphere bundle $SM$ (cf.\ \cite{valmfdsII}*{Definition 2.4.2}).
It consists of the oriented hyperplanes $[\xi]\in S_xM$ such that $\xi(v)\leq 0$ for all $v\in T_xA$. 
The normal cycle defines a closed Legendrian $(n-1)$-dimensional current on $SM$ 
(i.e.\ it vanishes on exact forms, and also on multiples of any contact form $\alpha\in\Omega^1(SM)$).

\begin{Definition}\label{def_val} Given complex-valued differential forms $\omega\in\Omega^{n-1}(SM)$, $\eta\in\Omega^n(M)$ the functional  $\lcur\omega,\eta\rcur$ defined by
 \begin{equation}\label{eq_def_val}
   \lcur\omega,\eta\rcur(A)=\int_{N(A)}\omega+\int_A\eta, \qquad A\in \calP(M)
 \end{equation}
is  a called a smooth valuation. The space of smooth valuations on $M$ is denoted by $\calV=\calV(M)$.
\end{Definition}
In case $M=V$ is a vector space, the space $\calV(V)^V$ of smooth translation invariant valuations coincides with $\Val^\infty(V)$. 

\subsubsection{Curvature measures}
\begin{Definition}\label{def_curv} Given complex-valued differential forms $\omega\in\Omega^{n-1}(SM)$, $\eta\in\Omega^n(M)$ the functional  $[\omega,\eta]$  which assigns to each $A\in \calP(M)$
the complex-valued Borel measure on $M$
 \begin{equation}\label{eq_def_curv}
  [\omega,\eta](A,U)=\int_{N(A)\cap \pi^{-1}(U)}\omega+\int_{A\cap U}\eta,\qquad U\subset M
 \end{equation}
 is called a smooth curvature measure.
The space of smooth  curvature measures on $M$ is denoted by  $\calC=\calC(M)$. 
\end{Definition}

We will denote by $\Curv=\Curv(V)$ the space of translation invariant smooth curvature measures, and by $\Curv_k$ its $k$-homogeneous component \cite{bfs}*{(2.1)}. We consider also the spaces $\Curv^+, \Curv^-$ defined by $\Phi\in \Curv^\varepsilon$ if and only $\Phi(-A,-U)=\epsilon\Phi(A,U)$.

Given a function $f\in C^\infty(M)$, to each curvature measure $\Psi$ we associate a  smooth valuation $\Psi_f$, also denoted $\Psi(f),\Psi(\cdot,f)$, by setting $\Psi_f(A)=\int_M f d\Psi(A,-)$.
In particular, for $f\equiv 1$ this defines the {\em globalization} map $$\glob:\calC(M)\to \calV(M)$$ The kernel of this map was described in \cite{bernig_brocker} in terms of the  Rumin differential \cite{rumin}. Given $\omega\in \Omega^{n-1}(SM)$, there exists an $(n-2)$-form $\xi$ such that
\begin{equation}
 D\omega:=d(\omega + \alpha\wedge\xi)
\end{equation}
is a multiple of the contact form $\alpha$. The $n$-form $D\omega$ is unique and is called the Rumin
differential of $\omega$. It was proved in \cite{bernig_brocker} that a curvature  measure $\Psi$ 
given by \eqref{eq_def_curv} satisfies $\glob\Psi=0$ if and only if
\[
 D\omega+\pi^*\eta=0\qquad \mbox{and}\qquad \int_{S_xM}\omega=0,\quad x\in M.  
\]

The spaces $\Omega^{n-1}(SM),\ \Omega^n(M)$ have a natural Fr\'echet topology. Following \cite{valmfdsII}*{Section 3.2}
we take on $\calV(M)$ and $\calC(M)$ the quotient topology defined by the maps $(\omega,\eta)\mapsto \lcur\omega,\eta\rcur$  and $(\omega,\eta)\mapsto [\omega,\eta]$. Since the kernel of these maps is closed, the spaces $\calC(M)$ and $\calV(M)$ are Fr\'echet.

\subsubsection{Pull-back}
If  $\iota\colon M_1\rightarrow M_2$ is an immersion then  the pull-back of valuations $\iota^*\colon \calV(M_2)\to \calV(M_1)$  and of curvature measures   $\iota^*\colon \calC(M_2)\to \calC(M_1)$  
is defined by (cf.\ \cite{valmfdsIG}*{Section 3.1} and \cite{bfs}*{Prop. 2.4})
\begin{align}
 &\iota^*\mu(A)=\mu(\iota(A)),\\ 
 &(\iota^*\Psi)_{\iota^*f}=\iota^*(\Psi_f),\label{pullback}
\end{align}
where $\mu\in\calV(M_2), \Psi\in\calC(M_2), A\in\mathcal P(M_1),$ and $f\in C^\infty(M_2)$. 

\subsubsection{Filtration and Euler-Verdier involution}
The space $\calV(M)$ is not naturally graded as in  \eqref{eq_mcmullen}, 
but there is a natural filtration $\calV(M)=\calV_0(M)\supset\cdots\supset\calV_n(M)$ defined by
$$\mu\in \calV_k \quad \Longleftrightarrow\quad \lim_{t\to 0}  \frac{1}{t^i} \mu(\phi(tA))=0,\quad i<k,\ A\in\calP(T_xM)$$
where $\phi:T_xM\rightarrow M$ is any local diffeomorphism with $\phi(0)=x$, cf.\ \cite{valmfdsII}.
Also, there is a decomposition of $\calV$, similar to the splitting of $\Val$ into even and odd parts,
induced by the  Euler-Verdier involution $\sigma\colon \calV\to\calV$ (cf.\ \cite{valmfdsII}). If $\varphi=\lcur \omega,\eta\rcur$, then  $\sigma \varphi=(-1)^n\lcur a^*\omega,\eta\rcur$ where $a\colon SM\to SM$ is the fiberwise antipodal map.
 This involution preserves the filtration. For $\varphi\in \Val_k$, one has $\sigma\varphi(A)=(-1)^k\varphi(-A)$. Similarly, there is a natural  involution $\sigma$ on $\calC(M)$ defined by $\sigma[\omega,\eta]=(-1)^n[a^*\omega,\eta]$.

\subsubsection{Algebra and module structures}
There is a product structure on $\calV(M)$, called the Alesker product, which turns this space into a commutative algebra with the Euler characteristic $\chi$ as unit.
\begin{Proposition}The Alesker product satisfies the following properties
\begin{enumerate}[i)]
 \item equivariance with respect to pull-backs: if $\iota\colon M_1\to M_2$ is an immersion, then $\iota^*(\mu_1\cdot\mu_2)=\iota^*(\mu_1)\cdot \iota^*(\mu_2)$.
 \item continuity, with respect to the natural topology on $\calV(M)$ (cf.\ \cite[Section 3.2]{valmfdsII}).
 \item compatibility with the filtration: $\calV_i\cdot\calV_j\subset \calV_{i+j}$
 \item compatibility with the Euler-Verdier involution: $\sigma(\mu_1\cdot\mu_2)=\sigma(\mu_1)\cdot\sigma(\mu_2)$.
\end{enumerate}
\end{Proposition}
The space $\calC(M)$  is naturally a module over the algebra $\calV(M)$ (cf.\ \cite[Proposition 2.3]{bfs}) with multiplication determined by
\[
 (\mu\cdot \Psi)_f=\mu\cdot \Psi_f,\qquad f\in C^\infty(M).
\]
It follows from \eqref{pullback} that
\begin{equation}\label{pullback_module}
 \iota^*(\mu\cdot \Psi)=\iota^*\mu\cdot\iota^*\Psi.
\end{equation}

\subsection{The frame bundle} Let $M$ be a  Riemannian manifold of dimension $n$.
The frame bundle $\pi\colon\mathcal F_{\mathrm O(n)}(M)\to M$ is a principal  bundle over $M$ with structure group $\mathrm O(n)$ whose fiber over each point $p\in M$ consists of all 
orthonormal bases of $T_pM$. The right action of $a\in \mathrm O(n)$ on $\mathcal F_{\mathrm O(n)}$ is denoted by $R_a$

Associated with the frame bundle  of $M$ are the $\R^n$-valued solder $1$-form $\tf$ and the $\mathfrak{so}(n)$-valued connection $1$-form $\cf$ on $\mathcal F_{\mathrm O(n)}(M)$.  At $e=(e_1,\ldots,e_n)\in\mathcal F_\mathrm{O(n)}(M)$ the solder form is defined by $d\pi_e=(e_1,\ldots, e_n)\cdot \omega_e$. 
The connection form is characterized by $d\tf=-\cf\wedge\tf$ (cf. \cite{bishop_crittenden}). Under the action of $\mathrm O(n)$ we have 
\begin{align}
 R_a^* \omega_i& = a_{ji}\, \omega_j \label{eq rg omega}\\
 R_a^* \varphi_{ij}& = a_{ki}\, a_{lj}\, \varphi_{kl} \label{eq rg varphi}
\end{align}

The solder and connection forms satisfy the  well-known structure equations 
\begin{align}
 d\tf_i&=-\cf_{ij}\wedge\tf_j \label{eq structure I}\\
d\cf_{ij}&=-\cf_{ik}\wedge\cf_{kj}+\Cf_{ij} \label{eq structure II}
\end{align}
where $1\leq i,j,k\leq n$, repeated indices are summed over,  and  the $\mathfrak{so}(n)$-valued curvature $2$-form $\Cf$ is
related to the curvature tensor $R(X,Y)Z=\nabla_{[X,Y]}Z-[\nabla_X,\nabla_Y]Z$ through
\[
 \Cf_{ij}= \frac1 2 \langle R(e_i,e_j)e_k,e_l\rangle\; \tf_k \wedge \tf_l.
\]

We will also make use of the following expressions for the pull-back of the solder and connection forms under a smooth map.

\begin{Lemma} \label{lemma pullback frame bundle} Let $N$ be a smooth manifold and let  $f\colon N \to \mathcal F_{\mathrm O(n)}(M)$ be a smooth map, $f(q)= (\pi(f(q)), e_1(q),\ldots, e_n(q))$. If $X$ is a tangent vector at $p\in N$, then 
\begin{align*}
f^*\tf_i(X)& = \langle e_i, d(\pi\circ f)(X)\rangle,\\
 f^* \cf_{ij} (X) & =  \langle e_i, \nabla_t ( e_j\circ \gamma)|_{t=0}\rangle,
\end{align*}
where $\nabla_t$ denotes the covariant derivative along any smooth curve $\gamma$ in $N$ with $\gamma(0)=p$ and $\dot\gamma(0)=X$. 
\end{Lemma}

\section{Filtrations and the Transfer Principle}\label{transfer}
Here we introduce a canonical filtration of $\calC(M)$ and prove Theorem~\ref{Theorem_graded_modules}. 
For a Riemannian manifold, Section~2.2.2 of \cite{bfs} describes a canonical isomorphism $\tau\colon \calC(M ) \to \Gamma(\Curv(T M ))$ called transfer map. Proposition~\ref{prop filtration transfer} clarifies 
the relation between the transfer map and the filtration.
Section~\ref{subsec transfer principle} contains applications of these ideas to the integral geometry of isotropic spaces. 

\subsection{Filtrations}
Given a smooth manifold $M$, let $\Val^\infty(TM)$ and $\Curv(TM)$ denote the Fr\'echet vector bundles over $M$ whose fibers over $x\in M$ are 
$\Val^\infty(T_xM)$ and  $\Curv(T_xM)$. 
We denote by $\Val^\infty_k(TM), \Curv_k(TM)$ the subbundles of $k$-homogeneous elements. 
Let $\glob\colon \Curv(TM)\to\Val^\infty(TM)$ denote the fiberwise globalization. The Alesker product and the module structure of $\Curv$ give rise to fiberwise products 
\begin{align*}
\Val^\infty(TM)\otimes\Val^\infty(TM)&\longrightarrow \Val^\infty(TM)\\
\Val^\infty(TM)\otimes\Curv(TM)&\longrightarrow\Curv(TM).
\end{align*}

Let us recall the construction of a canonical epimorphism $\Lambda_k\colon \calV_k\rightarrow \Gamma(\Val_k^\infty(TM))$ from \cites{valmfdsI,valmfdsII}.
If  $\mu\in\calV_k$ then  
\begin{equation}\label{def_Lambda}
\Lambda_k\mu|_x(A)=\lim_{t\to 0}\frac{1}{t^k}\mu(\phi(tA)), \qquad A\in\calP(T_xM), 
\end{equation}
where $\phi:T_xM\rightarrow M$ is any local diffeomorphism with $\phi(0)=x,\ d\phi_0=\mathrm{id}$, defines a translation-invariant valuation $\Lambda_k \mu|_x$ on $T_xM$. 
Since $\calV_{k+1}=\ker \Lambda_k$, this yields an isomorphism
 \begin{equation}\label{Xik}
  \Xi_k\colon\calV_k/\calV_{k+1}\rightarrow \Gamma(\Val_k^\infty(TM)).
 \end{equation}
Since the Alesker product is compatible with the filtration, the space $\bigoplus_k \calV_k/\calV_{k+1}$ is naturally a graded algebra.

\begin{Theorem}[\cite{valmfdsIV}]\label{graded_iso}
The graded map $\Xi:=\oplus_k\Xi_k$
\begin{equation}
\Xi\colon\bigoplus_k \calV_k/\calV_{k+1}\rightarrow \bigoplus_k \Gamma( \Val^\infty_k(TM))
\end{equation}
is an isomorphism of algebras. 
\end{Theorem}

Next we extend these constructions to the space of curvature measures. 
Let $\Omega^{n-1}_k(SM)$ denote the space of $\omega\in\Omega^{n-1}(SM)$ such that for all  $\xi\in SM$, one has $\omega_\xi|_E=0$ for any $(n-1)$-dimensional subspace $E\subset T_\xi SM$  such that $\dim E\cap \operatorname{ker}(d\pi) \geq n-k$. 
We introduce a filtration on $\calC(M)$  by
\begin{align*}\calC_k&=\{[ \omega,\eta]\colon\omega\in\Omega^{n-1}_k(SM),\eta\in\Omega^{n}(M)\},\quad k<n\\ \calC_n&=\{[ 0,\eta]\colon\eta\in\Omega^{n}(M)\}.\end{align*}
Notice that $\calV_k=\glob(\calC_k)$ by Proposition 5.2.5 of \cite{valmfdsI}.

Let us now construct a canonical map  $\Lambda_k'\colon \calC_k\to \Gamma(\Curv_k(TM))$. 
Let $x\in M$, and fix a local diffeomorphism $\phi:T_xM\to M$ with $\phi(0)=x,\ d\phi_0=\mathrm{id}$. 
For $t\in\R$ let $h_t(y)=ty, y\in T_xM$. For $\Psi\in\calC_k$ set
$$\Lambda_k'(\Psi)|_x:=\lim_{t\to 0}\frac{1}{t^k} (\phi\circ h_t)^*\Psi.$$

\begin{Proposition}\mbox{}\begin{enumerate}[i)]
                    \item  $\Lambda_k'\colon \calC_k\to \Gamma(\Curv_k(TM))$ is well-defined, independent of $\phi$ and surjective.
                    \item $\Psi\in \calC_k$ if and only if  at every $x\in M$ one has
                    \begin{equation}\label{eq_Ck}\lim_{t\to0} \frac{1}{t^{i}}  (\phi\circ h_t)^*\Psi =0, \qquad \text{for } i<k\end{equation}
                    for a local diffeomorphism $\phi:T_xM\to M$ with $\phi(0)=x,\ d\phi_0=\mathrm{id}$.
                    \item $\ker \Lambda_k'= \calC_{k+1}$
                   \end{enumerate}
\end{Proposition}

\begin{proof}
Let $\bar\phi,\bar h:ST_xM\to SM$  be the maps induced by $\phi,h$ on the cosphere bundle. 
Given $\Psi=[\omega,\eta]$ with $\omega\in\Omega_k^{n-1}(SM)$, $k<n$, and $A\in\mathcal P(T_xM), f\in C^\infty(T_xM)$ we have
\begin{align}
  \Lambda_k'(\Psi)_x(A,f)&=\lim_{t\to 0}\frac{1}{t^k}\int_{N(A)}f\cdot \bar h_t^*\bar\phi^*\omega +\frac{1}{t^k}\int_A h_t^*\phi^*\eta\notag\\
  &=\int_{N(A)}f\cdot  \pi_k(\bar\phi^*\omega)^{tr},\label{comp_Lambda}
\end{align}
where $(\bar\phi^*\omega)^{tr}$ denotes the translation-invariant form on $ST_xM$ which 
coincides with $\bar\phi^*\omega$ on every point of $\{0\}\times S_xM\subset ST_xM$, and $\pi_k$ 
takes the $k$-homogeneous part. In particular we see that $\Lambda_k'\Psi_x\in\Curv_k(T_xM)$.
Moreover, it is not difficult to check, e.g.\ in local coordinates, that $\pi_k(\bar\phi^*\omega^{tr})$ does not depend on $\phi$ and that $\Lambda'_k$ is surjective.

Repeating the argument leading to \eqref{comp_Lambda} shows that $\Psi\in \calC_k$ implies \eqref{eq_Ck}. Moreover, it is enough to prove the converse for $M=\R^n$. 
Let $\Psi=[\omega, 0]$ satisfy \eqref{eq_Ck} and let $i$ be maximal subject to $\omega\in \Omega^{n-1}_i$. If $i<k$, then repeating the computation leading to \eqref{comp_Lambda}
shows $[\pi_i(\omega^{tr}),0]=0$. Thus  $\omega$  can be modified by multiples of $\alpha$ and $d\alpha$ to yield a form $\tilde \omega\in \Omega^{n-1}_{i+1}$ 
with $\Psi=[\tilde \omega, 0]$;  a contradiction.

Item (iii) is an immediate consequence of (ii).
\end{proof}

It follows that $\Lambda_k'$ induces an isomorphism
\[
 \Xi'_k\colon \calC_k/\calC_{k+1}\rightarrow \Gamma(\Curv_k(TM)).
\]

\begin{Proposition}
 $\calV_k\cdot \calC_l\subset \calC_{k+l}$
\end{Proposition}
\begin{proof}
 Let $\mu\in\calV_k, \ \Psi\in \calC_l$. By \eqref{pullback_module} and the continuity of the Alesker product, we have for every $f\in C^\infty(T_xM)$ and $i\leq k+l$
\begin{align}\notag
\lim_{t\to 0} t^{-i}(\phi\circ h_t)^*(\mu\cdot\Psi)(f) &=\lim_{t\to 0} {t^{-i}} (\phi\circ h_t)^*\mu\cdot ((\phi\circ h_t)^*\Psi)_f\\
 &=\lim_{t\to 0} {t^{-k}} (\phi\circ h_t)^*\mu\cdot \lim_{t\to 0} t^{k-i} ((\phi\circ h_t)^*\Psi)_f. \label{eq_mu_psi}
 \end{align}
 If $i<k+l$, then the last factor vanishes, which shows $\mu\cdot \Psi\in\calC_{k+l}$.
\end{proof}

The graded space $\bigoplus_i\calC_i/\calC_{i+1}$ is thus naturally a module over $\bigoplus_i\calV_i/\calV_{i+1}$. 
This structure is compatible with the map $\Xi'=\oplus_k\Xi'_k$ in the following sense.
\begin{Proposition}Given $\varphi\in\bigoplus_i\calV_i/\calV_{i+1}$ and $\Theta\in\bigoplus_i\calC_i/\calC_{i+1}$
 \[
  \Xi'(\varphi\cdot\Theta)=\Xi(\varphi)\cdot\Xi'(\Theta).
 \]
\end{Proposition}
\begin{proof}Let $\varphi=\sum_i\mu_i$, and  $\Theta=\sum_j\Psi_j$ with $\mu_i\in\calV_i/\calV_{i+1}, \Psi_j\in\calC_j/\calC_{j+1}$. Equality \eqref{eq_mu_psi} gives 
\[
 \Xi_{i+j}'(\mu_i\cdot \Psi_j)=\Xi_i(\mu_i)\cdot\Xi_j'(\Psi_j),
\]
which directly yields the statement.
\end{proof}

Let now  $M$ be endowed with a connection (e.g.\ the Levi-Civita connection if $M$ is Riemannian).
Proposition 2.22 in \cite{bfs} defines a map $\tau:\calC(M)\to\Gamma(\Curv(TM))$ called {\em transfer map}, 
which associates to each $\Phi\in\calC(M)$ a section $\tau(\Phi)$ of $\Curv(TM)$ as follows. 
If $\Psi=[\omega,\eta]\in \calC(M)$, then $\tau(\Psi)|_x=[\omega^{tr},\eta^{tr}]$ 
where $\omega^{tr}\in\Omega^{n-1}(ST_xM)$ and $\eta^{tr}\in\Omega^n(T_xM)$ are the translation-invariant 
forms such that $\eta^{tr}_x=\eta_x$  and $\omega^{tr}_\xi= \omega_\xi$ for all $\xi\in{S_xM}$. 
Here we identified $T_\xi(SM)$ with $T_\xi(ST_xM)$ via the connection (cf.\ \cite{bfs}*{(2.9), (2.10)}).

\begin{Proposition}\label{prop filtration transfer}Let $\Psi\in\calC_k(M)$.
Then
 \[
  \Lambda_k'(\Psi)=\pi_k\circ \tau(\Psi),
 \]where $\pi_k\colon\Gamma(\Curv(TM))\to\Gamma(\Curv_k(TM))$ is the canonical projection.
\end{Proposition}
\begin{proof} For $k=n$ the proof is immediate. For $k<n$, let $\Psi=[\omega,\eta]$, and take $\phi=\exp_x$, the exponential map, in the definition of $\Lambda_k'$. By \eqref{comp_Lambda}, we see 
\[
 (\Lambda'_k\Psi)_x=[\pi_k(\bar\phi^*\omega)^{tr},0].
\]
But $(\bar\phi^*\omega)^{tr}$ is the translation-invariant form $\omega^{tr}$ defining $\tau(\Psi)$, since in this case the differential $(d\bar\phi)_\xi:T_\xi ST_xM\to T_\xi SM$ at  $\xi\in S_xM$ coincides with the identification mentioned before.
\end{proof}

\begin{Corollary}\label{coro_diagram}
The following diagram commutes
 \begin{displaymath}
 \xymatrix{\calC_k/\calC_{k+1} \ar[r]^-{\pi_k\circ\tau} \ar[d]_{\glob} &\Gamma(\Curv_k(TM))  \ar[d]^{\glob}
\\  \calV_k/\calV_{k+1}\ar[r]^-{\Xi_k}  & 
\Gamma(\Val_k^\infty(TM)) 
}
\end{displaymath}
\end{Corollary}

Suppose now further that $M$ is Riemannian.
The following lemma contains two basic properties that will be useful later.
\begin{Lemma}[cf. \cite{fu_bcn}*{eq. (2.4.8)}]\label{lemma_transfer}\mbox{}
\begin{enumerate}
 \item[i)]Let $\Psi\in \calC(M)$ such that $\tau(\Psi)\in\Gamma(\Curv_k(TM))$. Let $P\subset M$ be a  compact totally geodesic submanifold with boundary. If $\dim P\neq k$, then 
 $\Psi(P, \;\cdot\;)=0$ on $P\setminus\partial P$. 
 \item[ii)]Let $\Psi\in\calC_k(M)$. Let $P\subset M$ be a compact $k$-dimensional submanifold with boundary. Then
 \[
  \Psi(P,U)=\int_{P\cap U} \Kl_{\glob_x(\Psi)}(T_xP)dx, \qquad U\subset M,
 \]
where $\glob_x(\Psi)=\Lambda_k(\glob\Psi)_x\in\Val^\infty_k(T_xM)$, and $dx$ is the volume element on $P$. If $\dim P <k$, then $ \glob\Psi(P)=0$. 
\end{enumerate}
\end{Lemma}

\begin{proof}
 $i)$  For $k=n$ we clearly have $\Psi(P,\cdot)=0$ if $\dim P<n$. Suppose $k<n$, and let $\xi\in N(P)$ be such that $\pi(\xi)\in P\setminus \partial P$. Let $T_\xi SM=H\oplus V$  be the decomposition into horizontal and vertical parts induced by the connection. By assumption, $\Psi=[\omega,0]$  with
 \[
  \omega_\xi\in \largewedge^k H^*\otimes\largewedge^{n-k-1} V^*,
 \]
and $T_\xi N(P)=H_P\oplus V_P$ with $H_P\subset H$, $V_P\subset V$ and $\dim H_P=\dim P\neq k$. The statement follows.

$ii)$ Let $G, F$ be the bundles over $M$ whose fibers over $x$ are respectively $$G_x=\Gr_k(T_xM),\quad F_x=\{(\xi,E)\in S_xM\times \Gr_k(T_xM)\colon \xi\bot E\}.$$ Consider the double fibration $G\stackrel{\pi_1}{\leftarrow}F\stackrel{\pi_2}{\rightarrow} SM$ where $\pi_1,\pi_2$ are the obvious maps. 

Let $\Psi=[\omega,0]$. Given $P$ a $k$-dimensional submanifold, we consider $TP$ as a submanifold of $G$. Then 
$ N(P)\setminus \pi^{-1}\partial P=(\pi_2)_*\pi_1^*(TP)$ and thus
\[
\Psi(P,U)=\int_{TP|_U}(\pi_1)_*\pi_2^*\omega=\int_{P\cap U}f(T_xM) dx, 
\]
where $f$ is a smooth function on $G$.

Given $\phi$ as in \eqref{def_Lambda}, $E\in\Gr_k(T_xM)$ and $A\in\mathcal{K}(E)^{sm}$, the smoothness of $f$ implies
\[
 \glob(\Psi)(\phi(tA))=t^k\vol_k(A) f(E) +O(t^{k+1}).
\]
We conclude that $f=\Kl_{\glob_x(\Psi)}$, which proves the Lemma.
\end{proof}

\subsection{The transfer principle}\label{subsec transfer principle} A  Riemannian isotropic space $(M, G)$ is a Riemannian manifold $M$ with a Lie group of isometries $G$ acting transitively on the sphere bundle $SM$.
The spaces $\calV(M)^G, \calC(M)^G$ of $G$-invariant smooth valuations and curvature measures of $M$ are finite dimensional.

\subsubsection{Local, semi-local, global kinematic operators} We briefly recall these from Section~2.3.2 of \cite{bfs}. Respectively, they are the maps $K,\bar k, k$
\[
 \xymatrix{
 \calC^G \ar[r]^-{K} \ar[d]_{\mathrm{id}} &\calC^G\otimes \calC^G  \ar[d]^{\glob\otimes\mathrm{id}}\\ 
 \calC^G \ar[r]^-{\bar k} \ar[d]_{\glob} &\calV^G\otimes \calC^G  \ar[d]^{\glob\otimes\glob}\\ 
 \calV^G\ar[r]^-{k}  &  \calV^G\otimes \calV^G 
}
\]
uniquely determined by the commutativity of the diagrams and the defining relation:
\[
 K(\Psi)(A,U,B,V)=\int_G \Psi(A\cap gB,U\cap gV)\;dg
\]
for all $A,B\in\calP(M)$ and all Borel sets $U,V\subset M$. Here and in the following $dg$ 
is the Haar measure on $G$, normalized so that $dg\{g\in G\colon o\in g(U)\}=\vol(U)$
for any $o\in M$ and any Borel set $U\subset M$.

The space $\calV^G$ is spanned by finitely many valuations of the form
\[
 \mu_A^G=\int_G\chi(\;\cdot\;\cap gA)\;dg,\qquad A\in\calP(M).
\]
The module product of such a valuation with a curvature measure $\Psi\in\calC(M)$ is given by  (see \cite{bfs}*{Corollary 2.16})
\[
 \mu_A^G\cdot\Psi(B,U)=\int_{G}\Psi(B \cap gA,U)\;dg.
\]
If $M$ is compact, there is a pairing $\pd\colon\calV^G\otimes\calV^G\to \R$ called Alesker-Poincar\'e duality and given by
\begin{equation}\label{eq_pd}
 \pd(\varphi,\mu)=\frac{(\varphi\cdot\mu)(M)}{\vol(M)}.
\end{equation} This pairing was extended to non-compact spaces in \cite{bfs}. In all cases, $\pd$ is perfect and so defines an isomorphism $\pd\colon\calV^G\to{\calV^G}^*$.

The  {\em fundamental theorem of algebraic integral geometry} \cite{hig, bfs} states that the Alesker product and the global kinematic operator are dual to
each other under $\pd$; i.e.
\begin{equation}\label{ftaig}
(\pd\otimes\pd)\circ k=m^*\circ \pd 
\end{equation}
where $m^*$ is the adjoint of the Alesker product $m\colon\calV^G\otimes \calV^G\to\calV^G$. Equivalently, $k(\chi)=\pd^{-1}$ as elements in $\mathrm{Hom}((\calV^G)^*,\calV^G)$, and $k$ is multiplicative in the  sense that 
\[
k(\mu)=(\chi\otimes\mu)k(\chi)=(\mu\otimes\chi)k(\chi)
\]
for each $\mu\in\calV^G$. Similarly, for $\mu\in\calV^G$ and $\Psi\in \calC^G$, one has
\begin{align}
 K(\mu\cdot\Psi)&=(\chi\otimes\mu)\cdot K(\Psi)=(\mu\otimes\chi)\cdot K(\Psi),\\
 \bar k(\mu\cdot\Psi)&=k(\mu)\cdot(\chi\otimes\Psi).\label{semi_local_multiplicative}
\end{align}

\subsubsection{The transfer principle} Let $(M,G)$ be a Riemannian isotropic space, fix a representative point $o\in M$ and let $H$ denote its stabilizer.
When restricted to $G$-invariant curvature measures, 
the transfer map $\tau$ defines an isomorphism of vector spaces $\tau\colon \calC(M)^G\to\Curv(T_oM)^H$. 
The transfer principle (cf.\ \cite{howard,bfs}) states that the local kinematic operators in $(M,G)$ and $(T_oM,\overline H)$ correspond 
to each other through $\tau$. We will often omit $\tau$ and identify implicitly $\calC(M)^G$ with $\Curv(T_oM)^H$.

\begin{Lemma}\label{lemma_basis}Let $\glob\colon \Curv(T_oM)^H\cong\calC(M)^G\to \calV(M)^G$ be the globalization map in $M$. For each $k$, suppose that $\Psi_{k,1},\ldots,\Psi_{k,m_k}\in\Curv_k^H$ globalize
in $T_oM$ to a basis of $\Val_k^H$. Then $\{\glob(\Psi_{k,q})\colon 0\leq k\leq n, 1\leq q\leq m_k\}$ is a basis of $\calV(M)^G$.
\end{Lemma}
\begin{proof} By Corollary \ref{coro_diagram}, the restriction of $\glob$ to the space spanned by $\Psi_{k,1},\ldots,\Psi_{k,m_k}$ defines an isomorphism with $\calV_k^G/\calV_{k+1}^G$. It follows by induction that $\{ \glob(\Psi_{k,q})\}_{k,q}$ is a basis of $\calV(M)^G$.
\end{proof}

Consider now the euclidean vector space $V=T_oM$ under the  action of the affine group $\overline H=H\ltimes V$ generated by $H$ and translations. 
We denote by $dh$ the Haar probability measure on $H$. The following generalization of Proposition~5.2 of \cite{bfs} is a consequence of the transfer principle.

\begin{Proposition}\label{prop_geodesic}
Let $P\subset M$ be a compact totally geodesic submanifold containing $o$,  and suppose that the stabilizer of $P$ acts transitively on $P$. Consider the valuations $\mu_P\in\mathcal V(M)^{G}$ and $ \mu_{P_o}\in \Val^H(V)$ defined by
\begin{align*}
 \mu_{P}&={\mathrm{vol}(P)}^{-1}\int_G\chi(\,\cdot\,\cap gP)\,dg\\
 \mu_{P_o}&=\int_{H}\int_{hP_o^\bot} \chi(\, \cdot \,\cap (y+hP_o))\,dy\, dh
\end{align*}where $P_o=T_oP$.
After identifying  $\calC(M)^G$ and $\Curv^H$, we have for every $\Phi\in \Curv^H$
$$\mu_{P} \cdot \Phi = \mu_{P_o}\cdot \Phi.$$

\end{Proposition}
\begin{proof} Let $\glob_0\colon \Curv^H\to \Val^H$ and $k_0\colon \Val^H\to \Val^H\otimes \Val^H$ denote 
the globalization and global kinematic operator in $V=T_oM$.
By \cite{bfs}*{Prop. 2.17}, for every $\phi\in\mathcal V(M)^{G}$,
\[
 \langle\mathrm{pd} \phi, \mu_{P}\rangle = {\mathrm{vol}(P)}^{-1} \phi(P).
\]
By Lemma \ref{lemma_transfer}, given $\Phi\in\Curv^H$ we have 
\[
 \glob (\Phi)(P)=\mathrm{vol}(P)\Kl_{\glob_0\circ\pi_d\Phi}(P_o),
\]where $\pi_d\colon \Curv\rightarrow \Curv_d$ is the projection onto the homogeneous component of degree $d=\dim P$. Furthermore,
$$\Kl_{\glob_0\circ\pi_d\Phi}(P_o)=\lim_{R\to\infty}\frac{1}{\vol(B_R)}\glob_{0} (\Phi)(B_R),$$
where $B_R\subset P_o$ is a ball of radius $R$.

Now by \cite{bfs}*{Corollary 2.20}, writing  $K(\Phi)=\sum_{i,j}c_{ij}\Psi_i\otimes\Psi_j$, we have
\begin{align*}
 \mu_{P}\cdot \Phi&=\langle (\mathrm{id}\otimes(\mathrm{pd}\circ\mathrm{glob}))\circ K(\Phi),\mu_{P}\rangle\\
&=\sum_{i,j} c_{ij} \langle \mathrm{pd}\circ\glob \Psi_j,\mu_{P}\rangle\Psi_i\\
&=\sum_{i,j} c_{ij}\Kl_{\glob_0\circ\pi_d\Psi_j}(P_o)\Psi_i\\
& =\lim_{R\to\infty}\frac{1}{\vol(B_R)} (\mathrm{id}\otimes\glob_0)K(\Phi)(\;\cdot\;, B_R)\\
 & =   \left(\lim_{R\to\infty}\frac{1}{\vol(B_R)} k_0(\chi)(\;\cdot\;, B_R)\right) \cdot \Phi,
\end{align*}where the last equality uses \eqref{semi_local_multiplicative}.
Since 
$$ \lim_{R\to\infty}\frac{1}{\vol(B_R)} k_0(\chi)(\;\cdot\;, B_R)= \mu_{P_o},$$
the proof is complete. 
\end{proof}

\section{Simple valuations and curvature measures}\label{simple}

Recall that a  valuation $\psi$ on a vector space $V$  is called \emph{simple} if for every affine  hyperplane $E\subset V$. 
$$ \psi(K) =0, \qquad K\in \calK(E).$$

For translation-invariant and continuous valuations on $V$, the following characterization of simple valuations
was obtained by Klain \cite{klain_short} and Schneider 
\cite{schneider}.

\begin{Theorem}
If  $\mu\in \Val(\R^n)$ is simple then  there exist a constant  $c\in \R$  and an odd function $f\in C(S^{n-1})$ such that 
 $$\mu(K)= \int_{S^{n-1}} f \; dS_{n-1}(K)+ c\vol_n(K)$$
 for every $K\in \calK(\R^n)$. Conversely, every valuation of this form is simple.
\end{Theorem}

Here we establish an extension of the Klain-Schneider characterization to valuations that are not necessarily translation-invariant. We start by relating 
the notion of simplicity to the Alesker product of valuations.

\begin{Lemma}\label{lemma_simple}
 A valuation $\psi\in \Val^{\infty}(V)$  is  simple if and only if
 $$\phi\cdot \psi = 0$$
 for every $\phi\in \Val^{+,\infty}_1(V)$. 
\end{Lemma}
\begin{proof}  By \cite{alesker.bernstein}, every  $\phi\in \Val^{+,\infty}_1(V)$ can be written as 
  $$\phi(K) =  \int_{\AGr_{n-1}(V)} \chi(K\cap E) \; d\mu(E)$$
  with some smooth measure $\mu$ on the affine Grassmannian $\AGr_{n-1}(V)$. If $\psi$ is simple then 
  $$\phi\cdot \psi(K) = \int_{\AGr_{n-1}(V)} \psi(K\cap E) \; d\mu(E)=0.$$
  
  Conversely, assume that  $\phi\cdot \psi = 0$ for every $\phi\in \Val^{+,\infty}_1(V)$.  Let us  fix a euclidean structure on $V$. Let $u\in S(V)$ be arbitrary
  but fixed. We have to show that $\psi(K)=0$ for every convex body $K$ contained in $u^\perp$.  
  Choose  a family $(\rho_\varepsilon)_{\varepsilon\geq0}$ of non-negative functions
  on $S(V)$ with support shrinking to $u$  as $\varepsilon \to 0$ and $\int \rho_\varepsilon =1$. 
  Put  
  $$\phi_\varepsilon(L) = \int_{S(V)} \int_\R \chi(K\cap (v^\perp + t v)) \, dt\, \rho_\varepsilon(v)dv,\qquad L\in\calK(V).$$
  Theorem~1.3.4 of \cite{SchneiderBM} implies that if $t\neq 0,1$ then $v\mapsto\psi(K\times [0,u] \cap (v^\perp + t v))$ is continuous at $u$.  
  Thus 
  \begin{align*}0&= \phi_\varepsilon\cdot \psi (K\times [0,u]) \\
   & =  \int_{S(V)} \int_\R \psi(K\times [0,u] \cap (v^\perp + t v)) \, dt\, \rho_\varepsilon(v)dv \to \psi(K)
  \end{align*}
  as $\varepsilon \to 0$. 
  This completes the proof.  
\end{proof}

\begin{proof}[Proof of Theorem~\ref{Theorem_simple_valuations}]
To prove that $i)$ implies $ii)$ we fix some euclidean structure on $V$. Let $\phi\in \calV(V)$ be given by
$$\phi(A) = \int_{S(V)}\int_{\R} \chi(K \cap (u^\perp + tu)) f(u,t)\, dt\, du,$$
where $f$ is a smooth function. Since $\psi$ is simple, the same proof as in the translation-invariant case shows  $\phi\cdot \psi = 0$.

Since $\phi(\{p\})=0$ for $p\in V$, we have $\phi\in \calV_1$. Suppose $\psi \in \calV_k$ for $k<n-1$. As elements of the associated graded algebra 
$$\pi_k(\psi)\cdot \pi_1(\phi)= \pi_{k+1}(\psi \cdot \phi)=0.$$
Since $\Xi=\sum_k\Xi_k$ is an isomorphism of algebras, we have 
$$
 0  = \Xi_{k+1} (\pi_k(\psi)\cdot \pi_1(\phi))  = \Xi_k(\pi_k(\psi))\cdot \Xi_1(\pi_1(\phi)). $$
Now
$$\Xi_1(\pi_1(\phi))|_p (K) = \int_{S(V)} \vol_1(K|u)  f(u,\langle p, u\rangle )\; du,$$
where $\vol_1(K|u)$ denotes the length of the projection of $K$ on the subspace spanned by $u$.
Since $f$ was arbitrary, Lemma~\ref{lemma_simple} implies that $\Xi_k(\pi_k(\psi))$ is simple. By the Klain-Schneider characterization
of simple valuations, we must have $\Xi_k(\pi_k(\psi))=0$ and hence $\psi\in \calV_{k+1}$. 
We conclude that $\psi\in \calV_{n-1}$ and that $\Xi_{n-1}(\pi_{n-1}(\psi))$ is odd.

Since the Euler-Verdier involution is compatible with $\Xi$,  we see that $$\pi_{n-1}(\sigma\psi)=  (-1)^n \pi_{n-1}(\psi).$$
Since $\sigma$ is an involution, and it
acts on $\calV_n$ by multiplication by $(-1)^n$, we conclude that $\sigma(\psi)=(-1)^n\psi$. 

To show $ii)\Rightarrow iii)$ we note that pull-back preserves the filtration, so $\iota^* \psi \in \calV_{n-1}(N)$. If $\dim N<n-1$ we get $\iota^*\psi=0$. If $\dim N=n-1$, then $\sigma$ acts on $\calV_{n-1}(N)$ by multiplication by $(-1)^{n-1}$. Since $\sigma\iota^*\psi=\iota^*\sigma\psi=(-1)^n\iota^*\psi$, we deduce $\iota^*\psi=0$.

Since $iii)$ trivially implies $i)$, the proof is complete.
\end{proof}

\begin{Remark}
Alesker considered in \cite{valmfdsI} a filtration $\calV(V)=\gamma_0(V)\supset\cdots\supset\gamma_n(V)\supset \gamma_{n+1}(V)=\{0\}$ 
on $\calV(V)$ defined as follows: $\psi\in\gamma_i(V)$ if and only if $\psi|_E=0$ for all affine subspaces $E\subset V$ dimension $i-1$.
 The previous proofs can be easily adapted to show that $\psi\in\gamma_i(V)$ if and only if $\psi\in \calV_{i-1}(V)$ and $\Xi_{i-1}(\psi)_x\in \Val_{i-1}^-$ for all $x\in V$.  
\end{Remark}

We call a valuation $\psi\in \calV(M)$ on a manifold $M$  simple if $\iota^*\psi=0$ 
for every smoothly embedded manifold $\iota\colon N\hookrightarrow M$ with $\dim N<\dim M$.
\begin{Corollary}
A valuation $\psi$ on a manifold $M$ of dimension $n$  is simple if and only if $\psi\in \calV_{n-1}(M)$ and $\sigma(\psi)=(-1)^n\psi$.
\end{Corollary}
\begin{proof}
If $\psi$ is simple, then so is $\psi|_U$ for every open set $U\subset M$. 
Hence $\psi|_U\in\calV_{n-1}(U)$ and $\sigma\psi|_U=(-1)^n\psi|_U$ for every $U$  diffeomorphic to a vector space. This immediately 
implies $\psi\in\calV_{n-1}(M)$ (cf. \cite[Definition 3.1.1]{valmfdsII}) and $\sigma\psi=(-1)^n\psi$.

To show the converse, we can argue as in the last part of the proof of the previous theorem.
\end{proof}

We call a curvature measure $\Psi\in \Curv(V)$  simple if  $\iota_E^* \Psi =0 $ for every affine hyperplane $E\subset V$.
 To characterize simple curvature measures, the following lemma will be useful. The proof was communicated to us by Andreas Bernig.
\begin{Lemma}
        If $\Psi\in \Curv(V)$ is such that $\Psi_f=0$ for every affine function $f$ then $\Psi=0$.
       \end{Lemma}
       \proof
       Let $g$ be a continuous piecewise linear function: i.e. $g|_{T_i}$ is affine on every simplex $T_i$ of some locally finite triangulation of $V$. 
       Given a convex body $K$, we have
       $\Psi_g(K\cap T_i)=0$, and hence  $\Psi_g(K)=0$ by the  inclusion-exclusion principle (see, e.g., \cite{klain-rota}).
       Since every continuous function can be uniformly approximated by piecewise linear functions on compact domains, it follows that $\Psi=0$.
       \endproof

Let $\glob_k^{\pm} \colon \Curv_k^\pm(V) \to \Val_k^\pm(V)$ denote the globalization map.

\begin{Corollary} \label{simple_curvature_meas}
 A curvature measure $\Psi\in \Curv(V)$ is simple if and only if 
  $$\Psi \in \ker( \glob_{n-2}^+)\oplus   \Curv_{n-1}^-(V)\oplus \Curv_n(V). $$
\end{Corollary}
\begin{proof}
 If $f$ is a smooth function on $V$, then 
 $ \iota^*_E \Psi_f = (\iota_E^* \Psi)_{\iota^* f}=0$  for every affine hyperplane $E\subset V$, so $\Psi_f$ is simple.
The characterization of simple valuations implies that $\Psi_f\in \calV_{n-1}$ and that $\sigma(\Psi_f)=(-1)^n \Psi_f$. 
 
 If a curvature measure $\Phi\in \Curv_k(V)$ does not globalize to zero, then $\Phi_f\in \calV_k(V)\setminus \calV_{k+1}(V)$ for some $f$; if $\Phi\neq 0$ globalizes to zero, then 
 $\Phi_f\in \calV_{k+1}(V)\setminus \calV_{k+2}(V)$  for some $f$. In the first assertion it is enough to take $f$ constant. To check the second one, 
 take $f$ affine and note that $\Phi_f(tA)=t^{k+1}\Phi_{f-f(0)}(A)=t^{k+1}\Phi_f(A)$. It follows that $\Phi_f\in\calV_{k+1}$ and
 $\Phi_f=0$ if $\Phi_f\in  \calV_{k+2}.$ Since $\Phi\neq 0$, the previous lemma implies that there exists an affine $f$ such 
 that $\Phi_f\notin \calV_{k+2}$.
 
 Applying the preceding paragraph to $\Psi$, we obtain $$\Psi\in  \ker(\glob_{n-2})\oplus \Curv_{n-1}\oplus \Curv_n.$$ 
Decomposing $\Psi=\Psi^++\Psi^-$ into  its even and odd parts and assuming that $\Psi$ is homogeneous of degree $k$, we have $\sigma(\Psi_f)=(-1)^{k} \Psi^+_f +(-1)^{k+1} \Psi_f^-$ for every smooth function $f$ . 
Thus the identity  $\sigma(\Psi_f)=(-1)^n\Psi_f$ fixes the parity of $\Psi$. 
\end{proof}

Suppose now further that $V$ is euclidean.  
\begin{Corollary} \label{Corollary_invariant_restriction}
Let $G$ be a subgroup of $\SO(V)$ acting transitively on the sphere $S(V)$. Let $E\in \Gr_{n-1}( V)$ be a hyperplane and let $H$ be the stabilizer of $E$. The natural restriction map $\iota_E^*\colon \Curv^{G}\to \Curv^{H}$  has kernel $\ker\iota_E^*=(\ker\glob_{n-2}^+\oplus\Curv_n)\cap \Curv^G$. 
\end{Corollary} 
Note that $H$  need not be transitive on $S(E)$. An analogous statement at the level of valuations was proved in \cite[Corollary 3.5]{bernig_g2}. 
\begin{proof}Let $\Psi\in \Curv^G$ such that $\iota^*_E\Psi=0$. By transitivity and translation-invariance, $\Psi$ is simple. 
By Corollary~\ref{simple_curvature_meas} we have $\Psi\in \ker\glob_{n-2}^+\oplus\Curv_{n-1}^{-}\oplus\Curv_n$. But $\Curv_{n-1}^G$ contains only even elements, 
as $\Curv_0^G\cong \Curv^G_{n-1}$ is one-dimensional, cf.\ \cite{bernig_solanes}*{Proposition 3.1}.  
\end{proof}

\section{Integral geometry of $S^6$ under $\G_2$}

For $\lambda>0$ we  denote  by $\S^6$ the sphere of radius $\lambda^{-1/2}$ centered at the origin of $\im\O$. 
The first main result of this section is Theorem~\ref{first_part} below which states that $\calV^6_\lambda:=\calV(\S^6)^{\mathrm G_2}$ and
$\calV^6_0:=\Val(\C^3)^{\SU(3)}$ are isomorphic as filtered algebras. 
This proves the  first part of Theorem~\ref{Theorem_isomorphism}.
Our second main result is Theorem~\ref{Theorem_local_kinematic_SU3}, which establishes
 in explicit form the full array of local kinematic formulas for $\SU(3)$.

\subsection{Differential forms}
The map
\[
 J(v)=\sqrt{\lambda}\ p\cdot v,\qquad v\in T_{p}\S^6
\]
defines an almost-complex structure on $S_\lambda^6$, which is an isometry on each tangent space. Since the action of 
$\mathrm G_2$ preserves this complex structure it is called the homogeneous almost-complex structure on  $\S^6$ (and sometimes also the nearly K\"ahler structure).
Clearly, the associated hermitian form is 
\begin{equation}\label{eq Kahler form}\omega=\sqrt\lambda\, \iota_p \phi\end{equation}
where $\phi$ is the associative $3$-form  \eqref{eq associative 3-form}. Note that 
\begin{equation} \label{eq_omegae_not_closed} d\omega= 3 \sqrt{\lambda} \phi.\end{equation}

The restriction 
of $\phi$ to $\S^6$  is the real part of a complex volume form $\Upsilon$ on $T_p\S^6$. This may be seen by evaluating at the 
representative point $e_1/\sqrt\lambda$ and using \eqref{eq associative 3-form}.
Thus the stabilizer $H$ of the action of $\mathrm G_2$ on $S^6$ is contained in $\SU(3)$. 
In fact, from Proposition \ref{prop_transitive} we know that $(\S^6,\G_2)$ is an isotropic space and $H=\SU(3)$.

Let  $\mathcal F_{\SU(3)}\subset \mathcal F_{\mathrm O(6)}\S^6$ be the bundle whose fiber over $p\in \S^6$ consists of tuples 
$$(u_1, Ju_1,u_2, Ju_2, u_3, Ju_3)$$
where $u_1,u_2,u_3$ is a complex orthonormal basis of $T_p\S^6$ with 
\[
 \phi(u_1,u_2,u_3)=1.
\]
Clearly, $\mathcal F_{\SU(3)}$ is a principal bundle with structure group $\SU(3)\subset \mathrm{O}(6)$. We denote the restrictions of the solder $1$-form $\omega$ and the connection form $\varphi$
to $\mathcal F_{\SU(3)}$ by the same symbol, but we will use $1,\bar1,2,\bar 2, 3,\bar3$  as indices for the standard basis of $\R^6$ instead of 
$1,2,3,4,5,6$. 

It is natural to consider all the spaces $\S^6$ and $\C^3$ together, as a one-parameter family indexed
by the curvature $\lambda\geq0$. By means of a $\SU(3)$-frame at an arbitrary point $p\in\S^6$  we will identify $S^6_0=\C^3\cong T_p\S^6$.

The restriction of the connection $1$-form $\varphi$ to $\mathcal F_{\SU(3)}$ does not take values in $\mathfrak{su}(3)\subset \mathfrak{so}(6)$.

\begin{Proposition}\label{Prop S6 SU3}The restriction of the  connection $1$-form $\varphi$ to $\mathcal F_{\SU(3)}$ satisfies the following relations
\begin{align}
\cf_{1,\bar2}+\cf_{\bar1,2}&=-\sqrt{\lambda}\, \tf_3\label{first_rel}\\
\cf_{1,2}-\cf_{\bar1,\bar2}&=-\sqrt{\lambda}\,\tf_{\bar3}\label{second_rel}
\end{align}
and those obtained by cyclic permutation of the indices $1,2,3$.
Moreover
\begin{equation}
 \cf_{1,\bar1}+\cf_{2,\bar2}+\cf_{3,\bar3}=0\label{seventh_rel}
\end{equation}

\end{Proposition}
\begin{proof} Since $\nabla$, the Levi-Civita connection on $\S^6$, is induced by the covariant derivative of $\R^7$, one checks easily 
\begin{equation}
 \langle (\nabla_X J)Y,Z\rangle =\sqrt\lambda\, \iota_X\phi(Y,Z)
\end{equation}
with $(\nabla_X J)Y=\nabla_X (JY)-J\nabla_XY$.

If $u_1,u_{\bar1},\ldots, u_3,u_{\bar3}$ is a local $\SU(3)$-frame, then
\begin{align*}
 \langle u_1, \nabla_X u_{\bar 2}\rangle + \langle u_{\bar 1}, \nabla_X u_2\rangle & =\langle u_1, \nabla_X (Ju_{ 2})-J \nabla_X u_2\rangle\\
&=-\sqrt{\lambda}\,\phi(u_1,u_2,X)\\
&=-\sqrt{\lambda}\, \langle u_3, X\rangle ,
\end{align*}
and relation \eqref{first_rel} follows. Relation \eqref{second_rel} is proved similarly.
Differentiating the identity $\phi(u_{\bar1},u_2,u_3)=0$ yields
\begin{align*}
 0&=\phi(\nabla u_{\bar1},u_2,u_3)+\phi(u_{\bar1},\nabla u_2,u_3)+\phi(u_{\bar1},u_2,\nabla u_3)\\
&=\langle \nabla u_{\bar1},u_1\rangle-\langle \nabla u_2,u_{\bar2}\rangle-\langle\nabla u_3,u_{\bar3}\rangle,
\end{align*}
and proves \eqref{seventh_rel}.
\end{proof}

Let $\SU(2)\subset \SU(3)$ denote  the stabilizer of the first basis vector and consider in  $\mathcal{F}_{\SU(3)}$ the $1$-forms 
$$\alpha=\tf_1,\quad \beta=\tf_{\bar1},\quad \gamma=\cf_{\bar1,1}$$
and the $2$-forms 
\begin{align}
\theta_0&=\cf_{2,1}\wedge\cf_{\bar2,1}+\cf_{3,1}\wedge\cf_{\bar3,1}\label{def_theta0} \\
 \theta_1&=\tf_2\wedge\cf_{\bar2,1}-\tf_{\bar2}\wedge\cf_{2,1}+\tf_3\wedge\cf_{\bar3,1}-\tf_{\bar3}\wedge\cf_{3,1}\\
 \theta_2&=\tf_{2}\wedge\tf_{\bar 2}+\tf_{3}\wedge\tf_{\bar3}\\
\theta_s&=\tf_2\wedge\cf_{2,1}+\tf_{\bar2}\wedge\cf_{\bar2,1}+\tf_3\wedge\cf_{3,1}+\tf_{\bar3}\wedge\cf_{\bar3,1}\\
\chi_0&=(\cf_{2,1}\wedge\cf_{3,1}-\cf_{\bar 2,1}\wedge\cf_{\bar3,1})+\sqrt{-1}(\cf_{2,1}\wedge\cf_{\bar3,1}+\cf_{\bar2,1}\wedge\cf_{3,1})\\
\chi_1&=\tf_2\wedge\cf_{3,1}-\tf_{\bar2}\wedge\cf_{\bar3,1}-\tf_3\wedge\cf_{2,1}+\tf_{\bar3}\wedge\cf_{\bar2,1})\\
&\phantom{=}+\sqrt{-1}(\tf_{\bar2}\wedge\cf_{3,1}+\tf_2\wedge\cf_{\bar3,1}-\tf_3\wedge\cf_{\bar2,1}-\tf_{\bar3}\wedge\cf_{2,1})\\
 \chi_2&=(\tf_{2}\wedge\tf_{3}-\tf_{\bar2}\wedge\tf_{\bar3})+\sqrt{-1}(\tf_{2}\wedge\tf_{\bar3}+\tf_{\bar2}\wedge\tf_{3}).\label{def_chi2}
\end{align}
By \eqref{eq rg omega} and \eqref{eq rg varphi}, these forms are $\SU(2)$-invariant.
They are also horizontal for the bundle $ \mathcal{F}_{\SU(3)}\to S\S^6$ given by projection to the first basis vector. Thus they descend to the sphere bundle 
of $\S^6$. In the following we will make no notational distinction between the forms on the sphere bundle and those on the frame bundle trusting the reader 
to keep in mind where the identities are taking place.

For $\lambda=0$, these forms coincide with the ones introduced by Bernig in \cite{bernig_sun}.

\begin{Proposition}\label{differential} With the abbreviations  $\chi_{k,R}=\re(\chi_k)$ and $\chi_{k,I}=\im(\chi_k)$, the following identities hold:
\begin{align*}
d\alpha& =-\beta\wedge\gamma-\theta_s\\
d\beta& =\alpha\wedge\gamma+\theta_1-2\sqrt{\lambda}\chi_{2,R}\\
d\gamma&=2\theta_0-\sqrt{\lambda}\chi_{1,R}-\lambda\alpha\wedge\beta\\
d\theta_0&=-\sqrt{\lambda}\big(-\alpha\wedge\chi_{0,R}+\beta\wedge\chi_{0,I}+\gamma\wedge\chi_{1,I}\big)-\lambda\alpha\wedge\theta_1\\
d\theta_1&=2\alpha\wedge\theta_0+\gamma\wedge\theta_s-\sqrt{\lambda}\big(-\alpha\wedge\chi_{1,R}+2\beta\wedge\chi_{1,I}+2\gamma\wedge\chi_{2,I}\big)-2\lambda\alpha\wedge\theta_2\\
d\theta_2&=\alpha\wedge\theta_1+\beta\wedge\theta_s+ \sqrt{\lambda}\big(\alpha\wedge\chi_{2,R} - 3\beta\wedge\chi_{2,I}\big)\\
d\theta_s&=2\beta\wedge\theta_0-\gamma\wedge\theta_1 - \sqrt{\lambda}\big(\beta\wedge\chi_{1,R}-2\gamma\wedge\chi_{2,R}\big)\\
d\chi_0&=3\sqrt{-1}\gamma\wedge\chi_0+\sqrt{\lambda}\big((-\alpha+\sqrt{-1}\beta)\wedge\theta_0+\gamma\wedge(\theta_s+\sqrt{-1}\theta_1)\big)-\lambda\alpha\wedge\chi_1\\
d\chi_1&=2(\alpha+\sqrt{-1}\beta)\wedge\chi_0+2\sqrt{-1}\gamma\wedge\chi_1\\ &\phantom{=}+\sqrt{\lambda}\big((-\alpha+2\sqrt{-1}\beta)\wedge\theta_1
+\beta\wedge\theta_s+2\sqrt{-1}\gamma\wedge\theta_2\big)-2\lambda\alpha\wedge\chi_2\\
d\chi_2&=(\alpha+\sqrt{-1}\beta)\wedge\chi_1+\sqrt{-1}\gamma\wedge\chi_2-\sqrt{\lambda}(\alpha-3\sqrt{-1}\beta)\wedge\theta_2
\end{align*}

\end{Proposition}
\begin{proof}
 Direct verification using  the structure equations \eqref{eq structure I} and \eqref{eq structure II}, the identity 
 $\Phi_{ij} = \lambda \omega_i\wedge \omega_j$, and the relations of  Proposition~\ref{Prop S6 SU3}.
\end{proof}

 \begin{Proposition} The algebra of  $\mathrm G_2$-invariant forms on the sphere bundle of $\S^6$ is generated by the forms $\alpha,\beta,\gamma,
  \theta_0,\theta_1,\theta_2,\theta_s, \chi_0,\overline\chi_0,\chi_1,\overline\chi_1,\chi_2,  \overline\chi_2$. 
 \end{Proposition}
\begin{proof} These forms are $\mathrm G_2$-invariant by construction. The statement follows from the identification $\Omega(S\S^6)^{\G_2}\cong\Omega(S\C^3)^{\overline{\SU(3)}}$ and the 
description of translation and $\SU(3)$-invariant differential forms on $S\C^3$ given in \cite{bernig_sun}.
\end{proof}

\subsection{Curvature measures}
Next we describe the space $\Curv^{\SU(3)}$ of translation and $\SU(3)$-invariant curvature measures. 
By analogy with the case of valuations, we say that a $\SU(n)$-invariant curvature measure $\Phi\in\Curv^{\SU(n)}$ has weight $l$ if 
\[
 g^*\Phi=\det(g)^l\Phi\qquad \text{for}\ g\in \mathrm U(n).
\]
In particular $\Phi\in \Curv^{U(n)}$ iff $\Phi$ has weight $0$. Clearly, if $\Phi$ has weight 
$l$, then $\overline\Phi$ has  weight $-l$. Since $g^* \chi_i= \det(g)\chi_i$ while all other basic invariant forms are 
$\mathrm U(n)$-invariant, only curvature measures with weight $0,\pm 1,\pm 2$ exist.

Recall from \cite{bfs} that a basis of $\Curv^{\mathrm U(3)}$,  the space of curvature measures of weight $0$, is given by 
\begin{align}
 & \Delta_{0,0} \notag\\
 & \Delta_{1,0}, N_{1,0}\notag\\
 & \Delta_{2,0}, \Delta_{2,1}, N_{2,0} \label{eq_def_Delta_curv}\\
 & \Delta_{3,0}, \Delta_{3,1}, N_{3,1}\notag\\
 & \Delta_{4,1}, \Delta_{4,2}\notag\\
 & \Delta_{5,2}\notag\\
 & \Delta_{6,3}\notag
\end{align}
 Let us recall their explicit construction. Let $\omega_i$  be the volume of the $i$-dimensional unit ball, and let $c_{k,q}=(q!(3-k+q)!(k-2q)!\omega_{6-k})^{-1}$. Take
\begin{align*}
 B_{k,q}&=c_{k,q}[\beta\wedge\theta_0^{3-k+q}\wedge\theta_1^{k-2q-1}\wedge\theta_2^q,0]\\
 \Gamma_{k,q}&=\frac{c_{k,q}}{2}[\gamma\wedge\theta_0^{2-k+q}\wedge\theta_1^{k-2q}\wedge\theta_2^q,0]
\end{align*}
Then, the elements in \eqref{eq_def_Delta_curv} are given by
\[
 \Delta_{k,q}=\frac{1}{6-k}(2(3-k+q)\Gamma_{k,q}+(k-2q)B_{k,q}),\quad N_{k,q}=\frac{2(3-k+q)}{6-k}(\Gamma_{k,q}-B_{k,q}), 
\]
except for $\Delta_{6,3}(A,U)$ which is $\vol_6(A\cap U)$. 

The kernel of the globalization map on $\Curv^{\mathrm U(3)}$ for $\lambda=0$ is spanned by $N_{1,0},N_{2,0},N_{3,1}$. 

To find the  curvature measures of weight $1$ recall from \cite{bernig_sun}*{Proposition 3.4} that for every $i,j\in\{0,1,2\}$ one has
\[
\theta_i\wedge \chi_j\equiv c\ \theta_1\wedge \chi_1   \qquad\mod \theta_s
\]
for some constant  $c\in\C$. Hence, a curvature measure of weight $1$ must be a linear combination of the curvature measures 
\begin{equation} \label{eq_def_Psi_curv}
  \Psi_2:=\frac{1}{2\pi^2} [ \gamma\wedge \theta_1\wedge \chi_1, 0 ], \qquad \Psi_3:= \frac{1}{4\pi} [ \beta\wedge \theta_1\wedge \chi_1, 0 ].
\end{equation}
of degree $2$ and $3$.  
These curvature measures globalize to zero in the flat case since there are no valuations of weight $1$ in $\Val^{\SU(3)}$ (see \cite{bernig_sun}).

To find the  curvature measures of weight $2$, we note that
\[
 \chi_i\wedge \chi_j=0\qquad \mbox{ if } i+j\neq 2
\]
and that $\chi_1^2=-2\chi_0\wedge \chi_2$.
Thus, a curvature measure of weight $2$  must be a linear combination of 
\begin{equation}\label{eq_def_Phi_curv}
 \Phi_2:=\frac{1}{2\pi^2}[ \gamma\wedge \chi_0\wedge \chi_2, 0 ], \qquad \Phi_3:= -\frac{1}{4\pi} [ \beta\wedge \chi_0\wedge \chi_2, 0 ]
\end{equation}
which have degree $2$ and $3$ respectively. 
The curvature measure $\Phi_2$ globalizes to zero in the flat case since there is no valuation of weight $2$ in $\Val_2^{\SU(3)}$ (see \cite{bernig_sun}). We denote by $\phi_3=[\Phi_3]_0$ the globalization of $\Phi_3$ in $\C^3$. This valuation was denoted $\phi_2$ in \cite{bernig_sun}. 
By \cite{bernig_sun}*{Section 5.1}, its Klain function is 
\begin{equation}\label{klain_phi}
 \Kl_{\phi_3}(E)={\det}_\C(w_1,w_2,w_3)^2
\end{equation}
where $w_1,w_2,w_3$ is an orthonormal basis of $E\in\Gr_3(\C^3)$. We differ from \cite{bernig_sun} by a factor in the differential form defining this valuation. This is due to an error in Proposition 14 c) of \cite{bernig_sun} which should be $\Theta(W^\bot)=(-1)^{n^2/2}\Theta(W)$. 

Thus we have proved the following
\begin{Lemma}  A basis of $\Curv^{\SU(3)}$ is given by
\begin{align*}
 & \Delta_{0,0} \notag\\
 & \Delta_{1,0}, N_{1,0}\notag\\
 & \Delta_{2,0}, \Delta_{2,1}, N_{2,0},\Psi_2, \overline{\Psi_2},\Phi_2, \overline{\Phi_2} \label{eq_def_Delta_curv}\\
 & \Delta_{3,0}, \Delta_{3,1}, N_{3,1},\Psi_3,\overline{\Psi_3},\Phi_3,\overline{\Phi_3}\\
 & \Delta_{4,1}, \Delta_{4,2}\notag\\
 & \Delta_{5,2}\notag\\
 & \Delta_{6,3}\notag
\end{align*}
\end{Lemma}
 
Via the transfer principle identify $\calC_\lambda^6:=\calC(\S^6)^{\mathrm G_2}$ and $\Curv^{\SU(3)}$.  The globalization of a curvature measure $\Phi \in \Curv^{\SU(3)}$
in $\S^6$ will be denoted by $\glob_\lambda(\Phi)= [\Phi]_\lambda$.

\begin{Lemma}\label{lemma_euler_verdier} The Euler-Verdier involution on $\Curv^{\SU(3)}$ is given by 
$$\sigma \Phi = (-1)^k \Phi$$
for  $\Phi\in\Curv_k^{\mathrm U(3)}$ and 
\begin{align*}
  \sigma \Psi_2  & =    -\Psi_2 \\
  \sigma \Phi_2 & = \Phi_2 \\
  \sigma \Psi_3  & =    \Psi_3 \\
  \sigma \Phi_3  & =  -\Phi_3.\end{align*}
In particular $\Psi_2,\Psi_3\in \Curv^-$.
\end{Lemma}
\begin{proof}
The assignment 
$$ a(u_i)= (-1)^{i}u_i, \quad  a(u_{\bar i})=(-1)^{i}u_{\bar i}$$
 for  $i=1,2, 3$ defines a lift of  the fiberwise antipodal map 
$a\colon S \S^6\to S\S^6$ to the frame bundle $\mathcal{F}_{\SU(3)} $. 
One immediately verifies  that 
$$  a^* \tf_i = (-1)^i\tf_i, \quad  a^* \tf_{\bar i} = (-1)^i\tf_{\bar i} $$
and  
$$  a^* \cf_{i,1} = (-1)^{i+1}\cf_{i,1}, \quad   a^* \cf_{\bar i,1} = (-1)^{i+1}\cf_{\bar i,1} $$
 for  $i=1,2, 3$.
 
We conclude that 
$$  a^* \beta  =-\beta, \quad   a^* \gamma =\gamma,\quad   a^* \theta_i = (-1)^i \theta_i,\quad   a^* \chi_i  = (-1)^{i+1} \chi_i$$
for $i=0,1,2$. 
\end{proof}

\begin{Lemma}\label{lemma_conjugation}
The antipodal map $A\colon \S^6\rightarrow \S^6$ is anti-holomorphic {\em(}i.e. ${dA(Jv)=-J(dA(v))}${\em )}. Its action on 
$\Curv^{\SU(3)}$ is the following
\begin{equation}\label{eq_conjugation}
 A^*(\Delta_{k,q})=\Delta_{k,q},\quad A^*(N_{k,q})=N_{k,q},\quad A^*\Psi_i=\overline{\Psi_i},\quad A^*\Phi_i=\overline{\Phi_i}.
\end{equation}
\end{Lemma}
\begin{proof}
 For $v\in T_p\S^6$ we have
 \[
  dA(Jv)=dA(\sqrt{\lambda}p\cdot v)=-\sqrt{\lambda}p\cdot v.\]
 On the other hand, since $dA(v)\in T_{-p}\S^6$,
  \[
  J(dA(v))=J(-v)=(-\sqrt{\lambda}p)\cdot(-v)=\sqrt{\lambda}p\cdot v.
\]
We may lift the induced map $(A,dA)\colon S\S^6\rightarrow S\S^6$ to a map $A\colon \mathcal F_{\SU(3)}\rightarrow \mathcal F_{\SU(3)}$ as follows
\[
 A(p; u_1,u_{\bar 1},u_2,u_{\bar 2},u_3,u_{\bar 3})=(-p; -u_1,u_{\bar 1},u_2,-u_{\bar 2},-u_3,u_{\bar 3})
\]
One checks easily that
\[
A^*\omega_i=(-1)^{i+1}\omega_i,\quad A^*\omega_{\bar i}=(-1)^{i}\omega_{\bar i},
\]
and 
\[
A^*\varphi_{i,1}=(-1)^{i+1}\varphi_{i,1},\quad A^*\varphi_{\bar i,1}=(-1)^{i}\varphi_{\bar i,1}.
\]
Hence, 
\[
 A^*\beta=-\beta,\quad A^*\gamma=-\gamma,\quad A^*\theta_i=-\theta_i,\quad A^*\chi_i=-\overline{\chi_i},\qquad i=0,1,2.
\]Since $(A,dA)$ maps the normal cycle $N(K)$ to $N(A(K))$ with the orientation reversed, the identities \eqref{eq_conjugation} follow.
\end{proof}

\subsection{The Rumin differential}

The notion of weight, which was defined for curvature measures, can be also defined for differential forms on $S \S^6$ via transfer.  
It will become clear from the computation below, that this weight is not preserved by the Rumin differential. However, a certain pattern can be observed. 
Let $\Omega_{\epsilon}^{k,d-k}\subset \Omega^d(S\mathbb C^3)^{\SU(3)}\cong \Omega^d(S\S^6)^{G_2}$ denote the space of  invariant forms of bidegree $(k,d-k)$ and weight $\epsilon$. By Proposition \ref{differential}, if $\omega\in \Omega^{k,d-k}_{\epsilon}$, then
 \[
  d\omega=\eta_1+\sqrt{\lambda}\eta_2+\lambda\eta_3
 \]
where
\begin{align}\label{eta1}
  &\eta_1\in\Omega_{\epsilon}^{k,d-k+1},\\ \label{eta2}
  &\eta_2\in\Omega_{\epsilon-1}^{k+1,d-k}\oplus \Omega_{\epsilon+1}^{k+1,d-k},\\
  &\eta_3\in\Omega_{\epsilon}^{k+2,d-k-1}.\label{eta3}
 \end{align} Note also that $\eta_3$ is a multiple of $\alpha$.

\begin{Proposition}\label{prop_rumin_weight}
 For every $\omega\in\Omega^{k,n-k-1}_{\epsilon}$ we have
 \[
  D\omega=\rho_1+\sqrt{\lambda}\rho_2+\lambda\rho_3
 \]
with
\begin{align*}
& \rho_1\in \Omega^{k,n-k}_{\epsilon},\\
&\rho_2\in \Omega^{k+1,n-k-1}_{\epsilon-1}\oplus  \Omega^{k+1,n-k-1}_{\epsilon+1},\\
&\rho_3\in \bigoplus_{j=-2,0,2}\Omega^{k+2,n-k-2}_{\epsilon+j}.
\end{align*}
\end{Proposition}
\begin{proof}
Indeed, $D\omega=d(\omega+\alpha\wedge\xi)$, so 
\[
-d\alpha\wedge\xi\equiv d\omega\equiv \eta_1+\sqrt{\lambda}\eta_2
\] modulo $\alpha$, with $\eta_1,\eta_2$ as in \eqref{eta1},\eqref{eta2}. Hence, $\xi=\xi_1+\sqrt{\lambda}\xi_2$ with 
\[
\xi_1\in\Omega_{\epsilon}^{k-1,n-k-1},\qquad  \xi_2\in\Omega_{\epsilon-1}^{k,n-k-2}\oplus \Omega_{\epsilon+1}^{k,n-k-2}.
\]
Therefore
\[
 D\omega=d\omega+d\alpha\wedge\xi_1+\sqrt\lambda d\alpha\wedge\xi_2-\alpha\wedge d\xi_1-\sqrt\lambda\alpha\wedge d\xi_2,
\]
and the claim follows.
\end{proof}

We will need to compute explicitly the Rumin differential of a differential form underlying $\Delta_{2,1}$. In fact, the following information will be sufficient.

\begin{Lemma} \label{lemma du} Let $\omega=\gamma\wedge\theta_0\wedge\theta_2$. The bidegree $(4,2)$ part of $D\omega$ is given by 
\begin{multline*}\lambda \alpha\wedge \beta\wedge \left(4\chi_{1,I}^2-2\chi_{1,R}^2-\theta_0\wedge\theta_2-4\theta_1^2+\frac12\theta_s^2 \right) \\
 +\lambda\alpha\wedge \gamma\wedge  \left(4\chi_{1,I}\wedge\chi_{2,I}+4\chi_{1,R}\wedge\chi_{2,R}-5\theta_1\wedge\theta_2 \right)\end{multline*}
\end{Lemma}
\begin{proof}Our strategy mimics the algorithm described in \cite{bernig_solanes}*{Section 6}.
We need to solve for $\xi$ in 
\[
d\omega\equiv -d\alpha\wedge\xi,\qquad \mod \alpha.
\]
Using Proposition \ref{differential}, we get the following equivalence modulo $\alpha$
\begin{equation}\label{du}
 d\omega\equiv2\theta_0^2\wedge\theta_2+\beta\wedge\gamma\wedge\theta_0\wedge\theta_s
 -\sqrt{\lambda}(\beta\wedge\gamma\wedge\theta_2\wedge\chi_{0,I}+3\beta\wedge\gamma\wedge\theta_0\wedge\chi_{2,I}+\theta_0\wedge\theta_2\wedge\chi_{1,R})
\end{equation}
We want to express this form as a multiple of $d\alpha$. A simple computation shows $\theta_0^2\wedge\theta_2=-\theta_0\wedge\theta_s^2$. Hence, by the first equality in Proposition \ref{differential},
\begin{align*}
 2\theta_0^2\wedge\theta_2+\beta\wedge\gamma\wedge\theta_0\wedge\theta_s&=-2\theta_0\wedge(d\alpha+\beta\wedge\gamma)^2-\beta\wedge\gamma\wedge\theta_0\wedge d\alpha\\
 &=-5\beta\wedge\gamma\wedge\theta_0\wedge d\alpha-2\theta_0\wedge d\alpha^2\\
 &=d\alpha\wedge(-3\beta\wedge\gamma\wedge\theta_0+2\theta_0\wedge\theta_s).
\end{align*}

Combining equations (13) and (14) in \cite{bernig_sun} with $k=1,l=2$ yields 
\begin{align}
  \chi_1\wedge\theta_1\wedge\theta_s&=0\label{chi1_theta1_primitive}\\                                                                
  \chi_1\wedge(\theta_s^2-\theta_1^2)&=0 \label{chi1_theta1_squared}                                                                         
\end{align}  
From equations (11) and (12) of the same paper, we have
\begin{align*}
\theta_2\wedge\chi_{0,I}&=-\frac12\theta_1\wedge\chi_{1,I}+\frac12\theta_s\wedge\chi_{1,R}\\
\theta_0\wedge\chi_{2,I}&=-\frac12\theta_1\wedge\chi_{1,I}-\frac12\theta_s\wedge\chi_{1,R}.
\end{align*}
Hence, by \eqref{chi1_theta1_primitive}
\begin{align*}
 \beta\wedge\gamma\wedge\theta_2\wedge\chi_{0,I}&=-\frac{1}{2}\beta\wedge\gamma\wedge\theta_1\wedge\chi_{1,I}+\frac12\beta\wedge\gamma\wedge\theta_s\wedge\chi_{1,R}\\
 &=\frac12 d\alpha\wedge\theta_1\wedge\chi_{1,I}-\frac12\beta\wedge\gamma\wedge d\alpha\wedge\chi_{1,R}
\end{align*}
and similarly
\[
 \beta\wedge\gamma\wedge\theta_0\wedge\chi_{2,I}=\frac12d\alpha\wedge\theta_1\wedge\chi_{1,I}+\frac12\beta\wedge \gamma\wedge d\alpha\wedge\chi_{1,R}
\]
For the last term in \eqref{du}, we use again equations (11) and (12) of \cite{bernig_sun}, together with \eqref{chi1_theta1_squared}, to deduce
\[
 \theta_0\wedge\theta_2\wedge\chi_1=\frac14\chi_1\wedge(\theta_s^2+\theta_1^2)=\frac12\chi_1\wedge\theta_s^2.
\]
Hence,
\[
 \theta_0\wedge\theta_2\wedge\chi_{1,R}=\frac12\chi_{1,R}\wedge\theta_s^2=\frac12\chi_{1,R}\wedge(d\alpha^2+2\beta\wedge\gamma\wedge d\alpha)=\frac12 d\alpha\wedge\chi_{1,R}\wedge(\beta\wedge\gamma-\theta_s).
\]
Therefore, we can take
\[
 \xi=3\beta\wedge\gamma\wedge\theta_0-2\theta_s\wedge\theta_0 -\sqrt{\lambda}(-2\theta_1\wedge\chi_{1,I}-\frac32\beta\wedge\gamma\wedge\chi_{1,R}+\frac12\theta_s\wedge\chi_{1,R}).
\]
Using Proposition \ref{differential} to compute $ D\omega= d(\omega+\alpha\wedge\xi)$ completes the proof.
\end{proof}

\subsection{The algebra of invariant valuations} Let $\calV^6_\lambda:=\calV(\S^6)^{\G_2}$ if $\lambda>0$, and $\calV^6_0:=\Val(\C^3)^{\SU(3)}$. In this subsection we show that these algebras are isomorphic. We begin by introducing a convenient 
basis.

We will denote by $\glob_\lambda\colon \Curv^{\SU(3)} \cong \calC_\lambda^6 \to \calV_\lambda^6$ the globalization map. We will also use the notation 
$$[\Psi]_\lambda:= \glob_\lambda(\Psi).$$

\begin{Definition}
 For $\lambda\geq 0$ and $\max\{0,k-3\}\leq q\leq k/2\leq 3$ we set
 $$\mu_{k,q}^\lambda = [\Delta_{k,q}]_\lambda$$
 and 
 $$\phi_3^\lambda=[\Phi_3]_\lambda.$$
\end{Definition}
For $\lambda=0$ we  often write  $\mu_{k,q}, \phi_3$ instead of $\mu_{k,q}^0, \phi_3^0$. In this case, the classical intrinsic volumes are $\mu_k=\sum_q\mu_{k,q}$.

\begin{Corollary}
 For every $\lambda\geq0 $ the valuations $$\mu_{k,q}^\lambda, \phi_3^\lambda, \overline{\phi_3^\lambda}$$ with $\max\{0,k-3\}\leq q\leq k/2\leq 3$ constitute a basis
  of $\calV_\lambda^6$. 
\end{Corollary}
\begin{proof}
This is the case for $\lambda=0$ by \cite{bernig_sun}, and hence Lemma~\ref{lemma_basis} applies. 
\end{proof}

\bigskip 
For $\lambda>0$, we define the valuation $t_\lambda\in \calV^6_\lambda$ as 
the restriction of the first intrinsic volume $\mu_1$ (suitably normalized) on $\im\O$  to $\S^6$,
$$t_\lambda: = \left.\frac{2}{\pi} \mu_1\right|_{\S^6.}$$ For $\lambda=0$  we just take $t_0=t=\frac{2}{\pi}\mu_1$.

Let $\S^{5}$ be a totally geodesic subsphere of codimension one in $\S^6$, and consider the valuation
\[
 \varphi=\frac{2}{\sqrt\lambda}\frac{1}{\vol(\S^6)} k(\chi)(\,\cdot\,,\S^5).
\]
Let $\Delta_k=\sum_q\Delta_{k,q}$, and
\begin{equation}\label{eq tau} \tau_k =\frac{k!\omega_k}{\pi^k}[\Delta_k]_\lambda=\frac{k!\omega_k}{\pi^k} \sum_q \mu_{k,q}^\lambda, \end{equation}
where $\omega_k$ is the volume of the $k$-dimensional unit ball.
In particular $\tau_k=t^k$ if ${\lambda=0}$. Recall from the recast \cite{fu_bcn}*{\textsection 2.4.6} of classical spherical integral geometry into the language of the Alesker product
 that 
\begin{align}
 t^k_\lambda  & = \varphi^k \left(1- \frac{\lambda\varphi^2}{4} \right)^{-\frac k2}\label{eq tk_lambda},\\
 \varphi^k & = \sum_{j=0}^{\lfloor\frac{6-k}{2}\rfloor}  \left(\frac \lambda 4 \right)^j  \tau_{k+2j}, \label{eq phik}
\end{align}where the right hand side of \eqref{eq tk_lambda} should be understood as the truncated series expansion of the given function of $\varphi$.

Next we express the powers of $t_\lambda$ in terms of our basis of $\calV_\lambda^6$.

\begin{Lemma}\label{lemma_t_powers}

\begin{align}
\chi & = \mu^\lambda_{0,0}+\frac{\lambda}{2\pi}(\mu^\lambda_{2,0}+\mu^\lambda_{2,1}) \label{eq Euler characteristic}
+\frac{3\lambda^2}{4\pi^2}(\mu^\lambda_{4,1}+\mu^\lambda_{4,2})+\frac{15\lambda^3}{8\pi^3}\mu^\lambda_{6,3}\\
  t_\lambda&=\frac{2}{\pi}\mu^\lambda_{1,0}+\frac{3\lambda}{\pi^2}(\mu^\lambda_{3,0}+\mu^\lambda_{3,1})+\frac{15\lambda^2}{2\pi^3}\mu^\lambda_{5,2},\notag\\
 t_\lambda^2&=\frac{2}{\pi}(\mu^\lambda_{2,0}+\mu^\lambda_{21})+\frac{6\lambda}{\pi^2}(\mu^\lambda_{4,1}+\mu^\lambda_{4,2})+\frac{45\lambda^2}{2\pi^3}\mu^\lambda_{6,3},\notag\\
 t_\lambda^3&=\frac{8}{\pi^2}(\mu^\lambda_{3,0}+\mu^\lambda_{3,1})+\frac{40\lambda}{\pi^3}\mu^\lambda_{5,2},\notag\\
 t_\lambda^4&=\frac{12}{\pi^2}(\mu^\lambda_{4,1}+\mu^\lambda_{4,2})+\frac{90\lambda}{\pi^3}\mu^\lambda_{6,3},\label{t4}\\
 t_\lambda^5&=\frac{64}{\pi^3}\mu^\lambda_{5,2},\label{t5}\\
 t_\lambda^6&=\frac{120}{\pi^3}\mu^\lambda_{6,3}.\label{t6}
\end{align}
\end{Lemma}
\begin{proof}  These identities  follow immediately from \eqref{eq tau}, \eqref{eq tk_lambda}, and \eqref{eq phik}.
\end{proof}

Recall we are identifying $\calC_\lambda^6$ with $\Curv^{\SU(3)}$. 
There is a simple expression for the action of $t_\lambda$ on  $\Curv^{\SU(3)}$ 
in terms of the action of $t\in\Val^{\SO(6)}$.

\begin{Proposition}\label{prop_t_action}
Suppose $\Phi\in \Curv^{\SU(3)}$. Then 
$$t_\lambda \cdot \Phi = \left( t + \frac{\lambda}{8} t^3  + \frac{3\lambda^2}{128} t^5\right) \cdot \Phi.$$
\end{Proposition}

\begin{proof}
By Proposition \ref{prop_geodesic},
\[
 \varphi\cdot \Phi=c\ t\cdot\Phi,
\]
for some constant $c$ independent of $\Phi\in\Curv^{\SU(3)}$. Since $\glob_\lambda$ is injective on $\Curv^{\SO(6)}$,  equation (2.4.29) in \cite{fu_bcn} gives
\[
 \varphi\cdot\Delta_i=\Delta_{i+1}
\]
while obviously $t\cdot\Delta_i=\Delta_{i+1}$. This implies  $c=1$.
Invoking \eqref{eq tk_lambda} completes the proof.
\end{proof}

We set  $$u_\lambda:=\frac{2}{\pi}\mu^\lambda_{2,1}=\frac{2}{\pi} [\Delta_{2,1}]_\lambda.$$
 For $\lambda=0$ we  simply write $u$ instead of $u_0$, which agrees with \cites{hig, bfs}. 
 Using Proposition \ref{prop_t_action} and \cite[Proposition 5.10]{bfs} we obtain
\begin{align} 	
t_\lambda u_\lambda&=\frac{8}{3\pi^2}\mu^\lambda_{3,1}+\frac{8 \lambda}{5\pi^3}\mu^\lambda_{5,2},\label{tu}\\
t_\lambda^2 u_\lambda & = \frac{2}{\pi^2}  \mu_{4,1}^\lambda+\frac{4}{\pi^2} \mu_{4,2}^\lambda + \frac{6\lambda}{\pi^3}  \mu_{6,3}^\lambda,\label{t2u}\\
t_\lambda^3 u_\lambda& =\frac{64}{5\pi^3}\mu_{5,2}^\lambda, \label{t3u} \\
t_\lambda^4 u_\lambda&=\frac{24}{\pi^3}\mu_{6,3}^\lambda.\label{t4u}
\end{align}

\begin{Corollary}
The valuations $t_\lambda, u_\lambda, \phi_3^\lambda,\overline\phi_3^\lambda$ generate the algebra $\calV_\lambda$. Moreover the subalgebra generated 
by $t_\lambda, u_\lambda$ coincides with $\spann\{\mu_{k,q}^\lambda\}$. 
\end{Corollary}

The final ingredient  we need to determine the algebra structure of $\calV_\lambda^6$ is the following.

\begin{Proposition}\label{u2}
 $\pd(u_\lambda, u_\lambda)=\frac{18\lambda}{\pi^3}$.
\end{Proposition}
\begin{proof}
Recall  from \cite{bernig_product} that if $M$ is a compact riemannian manifold of dimension $n$ and   $\varphi_i\in \calV(M)$ is represented by $\omega_i \in\Omega^{n-1}(S M)$, 
$i=1,2$, then 
\[
 \pd(\varphi_1,\varphi_2)=\frac{1}{\vol(M)}\int_{SM} \omega_1\wedge a^*D\omega_2,
\]
where $a$ is the fiberwise antipodal map. 

Now $u_\lambda$ is represented by $\frac{1}{\pi^3} \gamma  \theta_0\theta_2$. The relevant part of the 
Rumin differential was already computed in Lemma~\ref{lemma du}.
A straightforward computation shows
\begin{align*}
 \alpha\wedge\beta\wedge\gamma\wedge\theta_0\wedge\theta_2\wedge\chi_{1,R}^2&=\alpha\wedge\beta\wedge\gamma\wedge\theta_0\wedge\theta_2\wedge\chi_{1,I}^2\\
 &=\alpha\wedge\beta\wedge\gamma\wedge\theta_0^2\wedge\theta_2^2=4\ d\!\vol_{S\S^6}\\
 \alpha\wedge\beta\wedge\gamma\wedge\theta_0\wedge\theta_2\wedge\theta_1^2&=\alpha\wedge\beta\wedge\gamma\wedge\theta_0\wedge\theta_2\wedge\theta_s^2=-4\ d\!\vol_{S\S^6}
\end{align*}
where $d\vol_{S\S^6}=d \vol_{\S^6}\wedge d\vol_{S^5}$ with $d\vol_{\S^6}=\frac12\alpha\wedge\beta\wedge\theta_2^2$ and $d\vol_{S^5}=\frac12\gamma\wedge\theta_0^2$. Using $\vol(S\S^6)=\pi^3\vol(\S^6)$ gives the result. 
\end{proof}

\begin{Corollary} \label{coro_u2}
\begin{align*}u_\lambda^2 & =  \frac{12}{\pi^2} \mu_{4,2}^\lambda + \frac{18\lambda}{\pi^3} \mu_{6,3}^\lambda\\
t_\lambda u_\lambda^2 & = \frac{64}{ 5 \pi^3} \mu_{5,2}^\lambda\\
t_\lambda^2u_\lambda^2 &= \frac{24}{\pi^3}\mu_{6,3}^\lambda.
\end{align*}

\end{Corollary}
\begin{proof}
By Theorem~\ref{graded_iso} the identity $u^2=\frac{12}{\pi^2}\mu_{4,2}$ (cf.\ \cite[(37), (39)]{hig}) implies 
$$u_\lambda^2\equiv \frac{12}{\pi^2}\mu_{4,2}^\lambda   \mod  \calV_5(\S^6)$$
In fact, this identity holds even modulo $\calV_6$ as can be seen from 
$$\pd(u_\lambda^2, t_\lambda)=0$$
which follows from the Euler-Verdier involution. Indeed, by Lemma \ref{lemma_euler_verdier},
\begin{align*}
u_\lambda^2t_\lambda (\S^6)&= -\sigma(u_\lambda)^2\cdot\sigma(t_\lambda) (\S^6)= -\sigma(u_\lambda^2t_\lambda) (\S^6)=-u_\lambda^2t_\lambda (\S^6),
\end{align*}
 since ${\sigma(\mu_{6,3}^\lambda)=\mu_{6,3}^\lambda}$.

Finally, we use the previous lemma to get the constant in front of $\mu_{6,3}^\lambda$. The proof of the second and third identities is even simpler. 
\end{proof}

The algebra structure of $\Val^{\mathrm U (n)}$ was determined by Fu in  \cites{fu_unitary}.
For $n=3$  
$$\Val^{\mathrm U(3)}\cong\C[t,u]/(f_4,f_5)$$
with $f_4(t,u)=t^4-6t^2u+u^2$ and $f_5(t,u)= t^5 -10 t^3u + 5tu^2$ where the isomorphism is given by 
$\frac{2}{\pi}\mu_{10}\mapsto t$ and $\frac{2}{\pi} \mu_{21}\mapsto u$.

\begin{Corollary}There is an algebra monomorphism $\Val^{\mathrm U(3)} \to \calV_\lambda^6$ with 
\begin{align} 
 t & \mapsto  t_\lambda - \frac{3\lambda}{8} t_\lambda^3\label{eq t image}\\
 u &\mapsto u_\lambda \label{eq u image}
\end{align}

\end{Corollary}
\begin{proof}
Define a morphism 
$g\colon\C[t,u]\to \calV_\lambda^6$ by \eqref{eq t image}, \eqref{eq u image}.
Using  equations \eqref{t4}, \eqref{t5}, \eqref{t6}, \eqref{t2u}, \eqref{t3u}, \eqref{t4u} and  Corollary~\ref{coro_u2},  we obtain
$$g(t^4-6t^2u+u^2)= g(  t^5 -10 t^3u + 5tu^2)=0.$$ Hence $g$
descends to a morphism $\Val^{U(3)}\to \calV_\lambda^6$. Since $g$ maps onto the subalgebra generated by
$t_\lambda$ and $u_\lambda$ and the dimensions agree,  $g$ must be injective.
\end{proof}

We can prove now the first part of Theorem \ref{Theorem_isomorphism}.

\begin{Theorem} \label{first_part} The filtered $\C$-algebra $\calV^6_\lambda$ is   isomorphic to 
$$\C[t,u,\phi_3,\overline{\phi_3}]/I$$
where $I$ is the ideal generated by 
$$f_4(t,u),f_5(t,u), t\phi_3,t\overline\phi_3,u\phi_3,u\overline\phi_3, \phi_3^2,\overline{\phi_3}^2,\phi_3\overline{\phi_3}+\frac{\pi^4 }{8}t^6.
$$ 
and the generators $t,u,\phi_3,\overline{\phi_3}$ have degrees $1,2$, and $3$ respectively. 
\end{Theorem}
\begin{proof}For $\lambda=0$ this is proved in \cite{bernig_sun}. 
Let us extend the morphism $g$ from the previous proof to a morphism from $\C[t,u,\phi_3,\overline{\phi_3}]$ to $\calV_\lambda^6$  
by $\phi_3\mapsto \phi_3^\lambda,\overline{\phi_3}\mapsto \overline{\phi_3^\lambda}$. 
To show that this descends to an isomorphism $\C[t,u,\phi_3,\overline{\phi_3}]/I\to\calV_\lambda^6$ 
what remains to be proven is
$t_\lambda\cdot\phi_3^\lambda=t_\lambda\cdot\overline{\phi_3^\lambda}=0$ and $u_\lambda\cdot\phi_3^\lambda=u_\lambda\cdot\overline{\phi_3^\lambda}=0$; 
the relations 
$$(\phi_3^\lambda)^2=(\overline{\phi_3^\lambda})^2=\phi_3^\lambda\overline{\phi_3^\lambda} +\frac{\pi^4}{8} t^6_\lambda=0$$
are by Theorem~\ref{graded_iso} and Corollary~\ref{coro_diagram} automatically satisfied.

Since $\Curv_4^{\SU(3)}$ contains no element of weight $2$, the module product $t\cdot\Phi_3$ must vanish. By Proposition \ref{prop_t_action}
we deduce $t_\lambda \cdot \phi_3^\lambda=0$. It remains only to show that $u_\lambda \cdot  \phi_3^\lambda=0$. Since $u\cdot\phi_3=0$ in $\Val^{\SU(3)}$, 
 we have  by Theorem \ref{graded_iso} and Corollary~\ref{coro_diagram}
$$ u_\lambda  \cdot \phi_3^\lambda\equiv 0  \mod \calV_6(S^6(\lambda)).$$
Since  $ \pd\left(u_\lambda, \phi_3^\lambda \right) =0$ by Euler-Verdier involution, the claim follows.
\end{proof}
It is interesting to note that the isomorphism $\calV_\lambda^6\cong  \Val^{\SU(3)}$ induced by the previous 
theorem preserves the filtration and Euler-Verdier involution.

\subsection{Global kinematic formulas} By the fundamental theorem of algebraic integral geometry \eqref{ftaig}, the full array of global kinematic formulas in $\S^6$ under $\mathrm G_2$
can be deduced from the  structure of the algebra of invariant valuations. We illustrate this by providing an explicit expression for the principal kinematic 
formula. We apply this result to bound the mean intersection number of Lagrangian submanifolds and arbitrary $3$-dimensional manifolds.

\begin{Theorem}\label{thm_pkf}
Writing $a\kfdot b=\frac12(a\otimes b+b\otimes a)$, the principal kinematic formula in $\S^6$ is given by
\begin{align*}
k_{\mathrm G_2}(\chi)&=2\mu^\lambda_{6,3}\kfdot\mu^\lambda_{0,0}+\frac{32}{15\pi}\mu^\lambda_{5,2}\kfdot\mu^\lambda_{1,0}
+\frac{5}{12}\mu^\lambda_{4,1}\kfdot\mu^\lambda_{2,0}\\
 &+\frac13\mu^\lambda_{4,1}\kfdot\mu^\lambda_{2,1}+\frac13\mu^\lambda_{4,2}\kfdot\mu^\lambda_{2,0}+\frac23\mu^\lambda_{4,2}\kfdot\mu^\lambda_{2,1}
 +\frac{2}{3\pi}\mu^\lambda_{3,0}\kfdot\mu^\lambda_{3,0}\\ & +\frac{8}{9\pi}\mu^\lambda_{3,0}\kfdot\mu^\lambda_{3,1}
 +\frac{16}{27\pi}\mu^\lambda_{3,1}\kfdot\mu^\lambda_{3,1}-\frac{2}{15\pi}\phi_3^\lambda\kfdot\overline{\phi_3^\lambda}\\
&+\frac{\lambda}{\pi}\mu^\lambda_{6,3}\kfdot(\mu^\lambda_{2,0}+\mu^\lambda_{2,1})+\frac{32\lambda}{15\pi^2}\mu^\lambda_{5,2}\kfdot(\mu^\lambda_{3,0}+\mu^\lambda_{3,1})\\
&+\frac{7\lambda}{24\pi}\mu^\lambda_{4,1}\kfdot\mu^\lambda_{4,1}
 +\frac{2\lambda}{3\pi}\mu^\lambda_{4,1}\kfdot\mu^\lambda_{4,2}+\frac{\lambda}{6\pi}\mu^\lambda_{4,2}\kfdot\mu^\lambda_{4,2}\\
 &+\frac{3\lambda^2}{2\pi^2}\mu^\lambda_{6,3}\kfdot(\mu^\lambda_{4,1}+\mu^\lambda_{4,2})+\frac{32\lambda^2}{15\pi^3}\mu^\lambda_{5,2}\kfdot\mu^\lambda_{5,2}
 +\frac{15\lambda^3}{8\pi^3}\mu^\lambda_{6,3}\kfdot\mu^\lambda_{6,3}. 
\end{align*}
\end{Theorem}

\begin{proof}
Recall that $k(\chi)=\pd^{-1}$ as elements in $\mathrm{Hom}((\calV_\lambda^6)^*,\calV_\lambda^6)$. 
Using Lemma \ref{lemma_t_powers}, Corollary \ref{coro_u2},   Theorem~\ref{first_part} and equations \eqref{tu}--\eqref{t4u} one gets all the pairings of elements from the basis $\chi,t,t^2,u,t^3,tu,\phi_3,\overline{\phi_3},t^4,t^2u,t^5,t^6$. 
Inverting the corresponding matrix we obtain $k(\chi)$ in terms of this basis. Using again Lemma \ref{lemma_t_powers} and equations \eqref{tu}--\eqref{t4u} yields the stated formula.  
\end{proof}

Given $E\in \Gr_3(\C^3)$ there exists $\theta\in [0,\frac\pi2]$  called the K\"ahler angle  of $E$ such that the hermitian form $\omega$ and the complex volume form $\Upsilon$ with $\re\Upsilon=\phi$ fulfill (cf. \cite{bernig_sun}*{(8)} or \cite{harvey_lawson}*{Theorem~III.1.7})
 \begin{equation}\label{kahler_angle}
 |\omega|_E|=\cos\theta,\qquad |\Upsilon|_E|=\sin\theta.
\end{equation}

A $3$-dimensional submanifold $L\subset \S^6$ is called \emph{Lagrangian} if 
$$JX \perp T_pL $$
for every $p\in L$ and every $X\in T_pL$.
Equivalently, Lagrangian submanifolds can be defined through one of the two conditions 
\begin{equation}\label{eq lagrangian}\omega|_L=0, \qquad 
 |\Upsilon|_L|=1.\end{equation}
Further,  every Lagrangian in $\S^6$ has constant phase in the sense  
\begin{equation} \label{eq slag}(\re \Upsilon) |_L=\phi|_L=0
\end{equation}
as $\omega, d\omega$ vanish on $L$ and so does $\phi$ by \eqref{eq_omegae_not_closed}.

\begin{Corollary} \label{cor application G2}
 Let $L\subset S^6$ be a compact Lagrangian submanifold.
 For every  $3$-dimensional  compact submanifold  with boundary $M\subset S^6$ the 
 following inequalities hold:
\begin{equation}\label{eq_application_G2}  \frac{4}{9\pi }\vol(L) \vol(M)  \leq   \int_{\mathrm G_2} \#( M\cap gL)\, dg < \frac{4}{5 \pi }\vol(L) \vol(M) .\end{equation}
 Equality holds on the left if and only if each tangent space of $M$ contains a complex line.
\end{Corollary}
\begin{proof}
Since $L$ is Lagrangian, by \cite{hig}*{eqs. (35) and (36)} and \eqref{klain_phi},
$$ \mu^\lambda_{3,0} (L) =  \vol(L) ,\qquad \mu_{3,1}^\lambda (L)=0, \qquad  \phi_3^\lambda (L) = -\vol(L).$$
Indeed, we have  
$$ \Kl_{\glob(\Delta_{3,0})_p}(E) = \sin^2\theta, \qquad \Kl_{\glob(\Delta_{3,1})_p}(E) = \cos^2\theta, \qquad \Kl_{\glob(\Phi_3)_p}(E) = \Upsilon(E)^2$$
where $\theta$ is the K\"ahler angle of $E\in \Gr_3(T_p S^6)$ and $\Upsilon(E)^2=\Upsilon(w_1,w_2,w_3)^2$ for some orthonormal basis of $E$.
The above identities follow now from Lemma~\ref{lemma_transfer} and equations \eqref{eq lagrangian}, and \eqref{eq slag}.

Plugging these values into the principal kinematic formula and using again Lemma~\ref{lemma_transfer} yields 
\begin{align*}\frac{1}{\vol(L)}\int_{\mathrm G_2} \#( M\cap gL)\, dg & 
 ={2}\left(\frac{1}{3\pi}\mu^\lambda_{3,0} +\frac{2}{9\pi} \mu^\lambda_{3,1} 
+{\frac{1}{15\pi}}\operatorname{Re} \phi_3^\lambda\right)(M)\\
& = \frac{2}{3\pi} \int_M  \sin^2 \theta +\frac{2}{3} \cos^2 \theta  
+{\frac{1}{5}} \re\Upsilon^2
\end{align*}
where $\theta= \theta(T_xM)$ denotes the K\"ahler angle  of $T_xM\subset T_x S^6$ and $\Upsilon^2=\Upsilon^2(T_xM )$.
Since $|\Upsilon|^2=\sin^2 \theta$, and $\re \Upsilon\equiv 0$ if $\theta\equiv \frac\pi2$, the inequalities follow.\end{proof}

Note that there are many 3-dimensional submanifolds $M\subset \S^6$ fulfilling the equality conditions in \eqref{eq_application_G2}. For instance, one may take a
non-singular almost complex curve $C$ in $\S^6$ (cf. \cite{complex_curves}) and a smooth family $\{g_t\}\subset \G_2$ such that $M=\bigcup_t g_t(C)\subset \S^6$ is a
$3$-dimensional  compact submanifold with boundary. Since $M$ is foliated by complex curves, every tangent space $T_pM$ contains a complex direction.

Note also that the upper bound in \eqref{eq_application_G2} is optimal. 
Indeed, let $p$ be a point on a 3-submanifold $N$ such that $\omega_p=0$ and $\Upsilon(T_pN)=1$ 
(e.g., $p=e_1$ and $T_pN=\spann(e_2,e_4,e_6)$), and take a sequence of relatively compact regions $M_r$ in $N$ converging to $p$. Then the previous proof shows
$$\lim_r\frac{1}{\vol(L)\vol(M_r)}\int_{\mathrm G_2} \chi( M_r\cap gL)\, dg = \frac{4}{5 \pi }.$$

\subsection{The kernel of the globalization map}

The goal of this section is to describe the kernel of the globalization map from $\Curv^{\SU(3)}$ onto $\calV_\lambda^6$.
We will use it to obtain a complete description of the local kinematic formulas for $\SU(3)$.

\begin{Corollary}
 $[N_{3,1}]_\lambda=0$.
\end{Corollary}
\begin{proof}Since $[N_{3,1}]_0=0$, by Corollary \ref{coro_diagram} and Euler-Verdier involution, 
we conclude that $[N_{3,1}]_\lambda$ is a multiple of $[\Delta_{5,2}]_\lambda$.  For a totally geodesic $\S^5\subset\S^6$, item $i)$ 
in Lemma \ref{lemma_transfer}, yields $[N_{3,1}]_\lambda(\S^5)=0$, while $[\Delta_{5,2}]_\lambda(\S^5)\neq 0$. The statement follows.
\end{proof}

To determine the kernel of the globalization map it will be sufficient to evaluate our invariant valuations 
on  geodesic balls of dimensions $4$ and $6$.

\begin{Lemma} \label{lemma_lagrangian_restriction}
If $\S^4\subset \S^6$ is a totally geodesic $4$-sphere, then
$$ [\Delta_{4,1}]_\lambda(\S^4) = \frac{32\pi^2}{15\lambda^2} \qquad \text{and}\qquad  [\Delta_{4,2}]_\lambda(\S^4) = \frac{8\pi^2}{15\lambda^2}.$$
If $B(r)\subset S_\lambda^4=  \S^6 \cap \langle e_1,\ldots, e_5\rangle$ is the geodesic ball of radius $0\leq r\leq \pi/\sqrt\lambda$ centered at $e_1$, then
  $$ \left [\Delta_{41} - 4 \Delta_{42} \right]_\lambda (B(r))  = -\frac{2\pi^2}{\lambda^2} \sin(\sqrt\lambda r)^4\, \cos(\sqrt\lambda r).$$
\end{Lemma}
\begin{proof}
Since $\mathrm G_2$ acts transitively on the Grassmannians of $2$ and $5$-dimensional spaces of $\im\mathbb O$, it makes no difference which $4$-dimensional totally geodesic sphere we choose. 
Suppose $\S^4 = \S^6 \cap \langle e_1,\ldots, e_5\rangle$. Since $\SU(2)\subset \mathrm G_2$
acts on $\S^4$ fixing $e_1$ and is transitive on $\S^4\cap \langle e_2,\ldots,e_5\rangle$  
it will be enough to consider the points $p=\lambda^{-1/2}(\cos\theta e_1 + \sin\theta e_2)$ on $\S^4$. With respect to the dual basis $e^6,e^7$ for the  basis $e_6,e_7$ of  $E^\bot=(T_p\S^4)^\bot$ we have $\omega|_{E^\bot}=\cos\theta\, e^6\wedge e^7$. Hence, by \cite[eqs. (35) and (36)]{hig},
\[
 \Kl_{\glob(\Delta_{4,1})_p}(E)= \sin^2\theta,\qquad \Kl_{\glob(\Delta_{4,2})_p}(E)= \cos^2\theta. 
\]
Finally,
\begin{align*}
 [\Delta_{4,1}]_\lambda (B(r))&=\vol( S^3)\int_0^{\sqrt\lambda r} \sin^2\theta \frac{\sin^3\theta}{\lambda^2}d\theta, \\
 [\Delta_{4,2}]_\lambda (B(r))&=\vol( S^3)\int_0^{\sqrt\lambda r} \cos^2\theta \frac{\sin^3\theta}{\lambda^2}d\theta. 
\end{align*}
\end{proof}

\begin{Lemma}\label{lemma_balls} 
Let  $\S^4= \S^6\cap \langle e_1,\ldots,e_5\rangle$ be a totally geodesic sphere and let 
$B(r)\subset S_\lambda^4$ be the geodesic ball of radius $0\leq r\leq \pi/\sqrt\lambda$ centered at $e_1$. 
 Then 
 \begin{align*}
 [N_{1,0}]_\lambda(B(r)) &  = -\frac{\pi}{\lambda^{1/2}}\cos(\sqrt\lambda r)^2\sin(\sqrt\lambda r)^3\\
 [N_{2,0}]_\lambda(B(r)) &  =  -\frac{\pi}{\lambda}\cos(\sqrt\lambda r) \sin(\sqrt\lambda r)^4\\
   [\Phi_{2}]_\lambda(B(r)) & =\frac{\pi}{ \lambda} \cos(\sqrt\lambda r)\sin(\sqrt\lambda r)^4\\
      [\Psi_2]_\lambda(B(r)) &  =\frac{4\pi}{3\lambda}\cos(\sqrt\lambda r)^2\sin(\sqrt\lambda r)^3\\
     [\Phi_{3}]_\lambda(B(r)) & =-\frac{\pi^2}{ \lambda^{3/2}} \sin(\sqrt\lambda r)^5  \\
  [\Psi_{3}]_\lambda(B(r)) & =\frac{\pi^2}{ \lambda^{3/2}}\cos(\sqrt\lambda r) \sin(\sqrt\lambda r)^4 \\
  [\Delta_{3,0}]_\lambda (B(r)) & = \frac{\pi^2}{ \lambda^{3/2}} \sin(\sqrt\lambda r)^5 \\
    [\Delta_{3,1}]_\lambda (B(r)) & = \frac{ \pi^2}{ \lambda^{3/2}} \cos(\sqrt\lambda r)^2\sin(\sqrt\lambda r)^3 
\end{align*}
\end{Lemma}

\begin{proof} Let $S_+^2=\{v\in S^2\colon v_1\geq 0 \}$ and consider   the map $F\colon S^3\times S_+^2\to N(B(r)) \cap  \pi^{-1} (\partial B(r) )$ given by 
$$F(u_1,u_2,u_3, u_4; v_1,v_2,v_3)=(q,\xi)$$ 
with 
\begin{align*} q & = \lambda^{-1/2}\left( \cos\rho\, e_1 + \sin\rho\, \sum_{i=1}^4 u_i e_{i+1} \right) \\
\xi & =v_1 \left( -\sin\rho\, e_1 + \cos\rho\, \sum_{i=1}^4 u_i e_{i+1} \right) +  v_2e_6 + v_3 e_7.
\end{align*}
where $\rho = \sqrt\lambda r$ for brevity. 

Let $\Phi\in \Curv_k^{\SU(3)}$ be represented by $\omega\in\Omega^5( S\C^3) \cong \Omega^5( S\S^6)$. By Lemma~\ref{lemma_transfer}, if $k\neq 4$, then 
$$[\Phi]_\lambda (B(r))= \int_{S^3\times S_+^2} F^*\omega.$$

Since $\SU(2)$ acts transitively on $\partial B(r)$ it is enough to compute $F^*\omega$ at a point $(u,v)$ with
\begin{equation}\label{eq special point} u_1=1, u_2=u_3=u_4=0 \qquad \text{and}  \qquad v_3=0.\end{equation}
Locally around this point choose a  lift $\bar F(u,v)= (b_1, Jb_1, b_2,Jb_2,b_3,Jb_3)$ of $F$ with values in $\mathcal{F}_{\SU(3)} (\S^6)$ such that
at \eqref{eq special point} we have 
\begin{align*}
b_1&=v_1 (-\sin\rho\,  e_1 + \cos\rho\,  e_2) + v_2 e_6\\
J b_1&  =  v_2 (- \sin\rho\,  e_4 + \cos\rho\,  e_7 ) +  v_1 e_3\\
b_2& =  v_1 (- \sin\rho\,  e_4 + \cos\rho\,  e_7 ) -  v_2 e_3\\
J b_2&= v_2 (-\sin\rho\,  e_1 + \cos\rho\,  e_2) - v_1 e_6\\
b_3&=e_5\\
Jb_3& = -\cos\rho\,  e_4 - \sin\rho\, e_7\end{align*}
Using Lemma~\ref{lemma pullback frame bundle} we  compute
\begin{align*}
 \bar F^* \omega_{\bar1} & = \lambda^{-1/2}\sin\rho\, (-v_2\sin\rho\, du_3 + v_1du_2)\\
 \bar F^* \omega_{2}  & = \lambda^{-1/2}\sin\rho\, (-v_1\sin\rho\, du_3 - v_2du_2)\\
 \bar F^* \omega_{\bar2} & =0\\
  \bar F^* \omega_{3}  & = \lambda^{-1/2}\sin\rho\, du_4\\
 \bar F^* \omega_{\bar3} & = - \lambda^{-1/2}\cos\rho\,\sin\rho\, du_3
\end{align*}
and 
\begin{align*}
 \bar F^*\varphi_{\bar1,1} & =\cos\rho\, v_2 dv_3 + v_1 \cos\rho\, (-\sin\rho\, v_2 du_3 +v_1 du_2) \\ 
 \bar F^*\varphi_{2,1}  & =v_1 \cos\rho\, dv_3 +  v_1 \cos\rho\, (-\sin\rho\, v_1 du_3 -v_2 du_2)\\
 \bar F^*\varphi_{\bar 2,1} & =  v_2 dv_1 -v_1 dv_2 \\
 \bar F^*\varphi_{3,1}  & =v_1\cos\rho\, du_4\\
 \bar F^*\varphi_{\bar3,1} & = -\sin\rho\, dv_3 -v_1 \cos^2 \rho\, du_3
\end{align*}

Let us use these relations to compute for example $[\Delta_{3,0}]_\lambda(B(r))$. Substituting 
$$v_1=\cos\varphi,\quad v_2 = \sin\varphi, \quad dv_3= \sin\varphi d\theta$$
we obtain 
\begin{align*}[\Delta_{3,0}]_\lambda(B(r)) & = \frac{1}{8\pi} \int_{S^3\times S_+^2}  \bar F^*( \beta \wedge  \theta_1 \wedge \theta_1)\\
& = \frac{1}{8\pi}\vol(S^3)\int_0^{2\pi}\int_0^{\frac{\pi}{2}} \frac{2}{\lambda^{\frac32}}\sin(\sqrt\lambda r)^5\sin\varphi   d\varphi d\theta,  
\end{align*}which yields the stated value. The other identities follow similarly. 
\end{proof}

\begin{Lemma} \label{lemma_6-balls}
Let $B(r)\subset \S^6$ be a 6-dimensional geodesic ball. Then
the curvature measures $N_{1,0}$, $\Psi_2$, $\Phi_3$ and $\Delta_{3,0}-\frac23\Delta_{3,1}$ vanish identically on $B(r)$.
\end{Lemma}
\begin{proof}Since $\G_2$ acts transitively on $\partial B(r)$, it is enough to prove that the globalizations $[N_{1,0}]_\lambda$, $[\Psi_2]_\lambda$, $[\Phi_3]_\lambda$ and $[\Delta_{3,0}-\frac23\Delta_{3,1}]_\lambda$ 
vanish on $B(r)$. Let us show this first for $\lambda=0$. In this case $N_{1,0},\Psi_2$ fulfill the statement trivially as they globalize to 0. 
To deal with $\phi_3$ and $\mu_{3,0}-\frac23\mu_{3,1}$ we consider the derivation operator $\Lambda\colon \Val_k\rightarrow \Val_{k-1}$ introduced in \cite{alesker_hard_lefschetz}. 
Since there are no elements of weight 2 in $\Val_2^{\SU(3)}$ we must have $\Lambda\phi_3=0$. By \cite[Lemma 5.2]{hig} we have $\Lambda^2(\mu_{3,0}-\frac23\mu_{3,1})=0$. 
We deduce that $\phi_3(B(r))$ and $(\mu_{3,0}-\frac23\mu_{3,1})(B(r))$ are polynomials in $r$ of respective degrees 0 and 1 at most. On the other hand, since $\phi_3, \mu_{3,0},\mu_{3,1}$ are homogeneous of degree 3, these polynomials must be proportional to $r^3$. This proves the statement for $\lambda=0$.

For a general $\lambda$ let us note that $\partial B(r)$ is totally umbilical. Hence, for any curvature measure $\Psi\in \Curv_k^{\SU(3)}$  with $k<6$ we have
\[
 [\Psi]_\lambda(B(r))=[\Psi]_0(B_0(r))
\]
where $B_0(r)$ is any ball in $\C^3$ with the same normal curvature as $B_r\subset\S^6$. Thus, the lemma follows from the case $\lambda=0$.
\end{proof}

\begin{Proposition}\label{prop_ker_glob} The kernel of the globalization map $\glob_\lambda\colon \Curv^{\mathrm{SU}(3)}\rightarrow \mathcal \calV_\lambda^6$ is given by the following relations:
 \begin{align*}
 [N_{2,0}]_\lambda & = \frac{\lambda}{2\pi}\left [\Delta_{4,1} - 4 \Delta_{4,2} \right]_\lambda\\
  [\Phi_2]_\lambda & = -\frac{\lambda}{2\pi} \left [\Delta_{4,1} - 4 \Delta_{4,2} \right]_\lambda\\
  [\Psi_3]_\lambda & = -\frac{\sqrt\lambda}{2}\left [\Delta_{4,1} - 4 \Delta_{4,2} \right]_\lambda \\
   [N_{1,0}]_\lambda&=\frac{3\lambda}{2\pi}\left[\Delta_{3,0}-\frac{2}{3}\Delta_{3,1}+\re\Phi_3\right]_\lambda\\
    [\Psi_2]_\lambda & = -\frac{2\sqrt\lambda}{\pi} \left [\Delta_{3,0} -\frac{2}{3}\Delta_{3,1} + \Phi_3 \right]_\lambda\\
     [N_{3,1}]_\lambda&=0
 \end{align*}
In particular $N_{3,1},\im\Phi_2,\im\Psi_3, N_{2,0}+\Phi_2\in\ker\glob_\lambda$ for all $\lambda$. 
\end{Proposition}

\begin{proof} From  $[N_{2,0}]_0=[\Phi_2]_0=[\Psi_3]_0=0$, Corollary \ref{coro_diagram}, Lemma \ref{lemma_euler_verdier},  Lemma~\ref{lemma_transfer} 
and the first part of Lemma~\ref{lemma_lagrangian_restriction} we conclude that  
$  [N_{2,0}]_\lambda$, $[\Phi_2]_\lambda$,  $[\Psi_3]_\lambda$, and $\left [\Delta_{4,1} - 4 \Delta_{4,2} \right]_\lambda$ must be proportional.
From Lemma~\ref{lemma_balls} and the second part of Lemma~\ref{lemma_lagrangian_restriction} we get the proportionality factors.

From  $[\Psi_{2}]_0=0$,  Euler-Verdier involution,  and Lemma~\ref{lemma_transfer} we conclude
\begin{align*}
 [\Psi_2]_\lambda&=a[\Delta_{3,0}]_\lambda+b[\Delta_{31}]_\lambda+c[\Phi_3]_\lambda+d[\overline{\Phi_3}]_\lambda.
\end{align*}
Since the Rumin differential $D$ must agree on both sides, and looking at the weight
of each curvature measure, it follows from Proposition~\ref{prop_rumin_weight} that $d=0$.
By Lemma \ref{lemma_balls}, $b=\frac{4\sqrt\lambda}{3\pi}$ and $a-c=0$.
Finally $2a+3b=0$ by  Lemma \ref{lemma_6-balls}.

Arguing as before, using $[N_{1,0}]_0=0$, Euler-Verdier involution,  Lemmas~\ref{lemma_transfer}, \ref{lemma_balls}, and \ref{lemma_6-balls}, we obtain
\[
 [N_{1,0}]_\lambda=\frac{3\lambda}{2\pi}[\Delta_{3,0}]_\lambda-\frac{\lambda}{\pi}[\Delta_{3,1}]_\lambda+c [\Phi_3]_\lambda+d[\overline{\Phi_3}]_\lambda
\]
with $c+d=\frac{3\lambda}{2\pi}$. Finally, Lemma \ref{lemma_conjugation} implies $c=d$.  This completes the proof.
\end{proof}

\subsection{Local kinematic formulas}

The local kinematic formulas for the group  $\mathrm U (n)$ were recently discovered in \cite{bfs}. Our results on the integral 
geometry of $\S^6$ are enough to completely determine  the local kinematic formulas  for $\SU(3)$. In the proof below, we will use the local kinematic formulas
\begin{align}\label{eq local Delta0}
K_{\mathrm U(3)} &(\Delta_{0,0})= 2\Delta_{0, 0}\kfdot\Delta_{6, 3}+\frac{32}{15\pi}\Delta_{1, 0}\kfdot\Delta_{5, 2}+\frac{5}{12}\Delta_{2, 0}\kfdot\Delta_{4, 1}+\frac13\Delta_{2, 0}\kfdot\Delta_{4, 2}\notag\\ 
&+\frac13\Delta_{2, 1}\kfdot\Delta_{4, 1}+\frac23\Delta_{2, 1}\kfdot\Delta_{4, 2}+\frac1{24}N_{2, 0}\kfdot\Delta_{4, 1}-\frac16 N_{2, 0}\kfdot\Delta_{4, 2}\\ \notag 
& +\frac{2}{3\pi}\Delta_{3, 0}\kfdot\Delta_{3,0}+\frac{8}{9\pi}\Delta_{3, 0}\kfdot\Delta_{3, 1}+\frac{16}{27\pi}\Delta_{3, 1}\kfdot\Delta_{3,1}\\ \notag&-\frac{32}{675\pi} N_{3, 1}\kfdot N_{3,1}+\frac{8}{45\pi}\Delta_{3, 0}\kfdot N_{3, 1}-\frac{16}{135\pi}\Delta_{3, 1}\kfdot N_{3, 1}\\
K_{\mathrm U(3)} &(\Delta_{2,0}+\Delta_{2,1}) =   2 \Delta_{6,3}\kfdot (\Delta_{2,0}+\Delta_{2,1}) \label{eq local Delta2} 
+ \frac{64}{15\pi } \Delta_{5,2}\kfdot  (\Delta_{3,0}+\Delta_{3,1}) \\
&+ \frac{29}{48} \Delta_{4,1}\kfdot \Delta_{4,1}  + \frac 76 \Delta_{4,1}\kfdot \Delta_{4,2} \notag
+ \frac 23  \Delta_{4,2}\kfdot \Delta_{4,2}\\
 K_{\mathrm U(3)}&(\Delta_{3,0} - \frac23\Delta_{3,1} )   = 2 \Delta_{6,3}\kfdot (\Delta_{3,0} -\frac23 \Delta_{3,1} )  + \frac{4}{15} \Delta_{5,2}\kfdot ( \Delta_{4,1} - 4 \Delta_{4,2})
 \label{eq local Delta3}
\end{align}

\begin{Theorem} \label{Theorem_local_kinematic_SU3}  The local kinematic formulas for $\SU(3)$ are given by
  \begin{align} K_{\SU(3)}(\Delta_{0,0} ) & = K_{\mathrm U(3)}(\Delta_{0,0} )    -\frac{2}{15\pi}  \Phi_3\kfdot \overline{\Phi}_3  
 -\frac{1}{ 8\pi}\Psi_3\kfdot \overline{\Psi}_3 \notag\\
  K_{\SU(3)} (\Phi_2) & = 2\Phi_2 \kfdot\Delta_{6,3} \notag\\
   K_{\SU(3)} (\Psi_2) & = 2 \Psi_2\kfdot\Delta_{6,3} + \frac{ 16}{15\pi} \Psi_3\kfdot \Delta_{5,2} \notag\\
   K_{\SU(3)} (\Phi_3) & = 2\Phi_3\kfdot \Delta_{6,3} \label{eq local Phi_2}\\
   K_{\SU(3)} (\Psi_3) & = 2 \Psi_3\kfdot\Delta_{6,3}  \label{eq local Psi_2}
 \end{align}
 and 
 \begin{equation}\label{SU3_local}K_{\SU(3)}  = K_{\mathrm U(3)}\qquad \text{on} \ \Curv_i^{\mathrm U(3)}\end{equation}
 for $1\leq i\leq 6$.

\end{Theorem}

\begin{proof}
Applying the  Fubini theorem for integration on homogeneous spaces
$$\int_G f\;  dg= \int_{G/H}\int_H  f(xh) \, dh\, dx$$
 to $G=\mathrm U(3)$ and $H=S^1$ yields
\begin{equation}\label{eq local kinematic SU3} K_{\mathrm U(3)}(m)(A,U; B,V) = \int_{S^1}  K_{\SU(3)}(m)(A,U; \zeta B, \zeta V)\;d\zeta.\end{equation}

 Let $\{m_j\} $ be some basis of $\Curv_{\SU(3)}$, where each $m_j$ has pure degree of homogeneity and pure weight. 
 Let also $m$ be a curvature measures of pure 
 degree and weight. If 
 $$K_{\SU(3)} (m) = \sum_{i,j} c^m_{ij} m_i\otimes m_j,$$
 then $c^m_{ij}\neq 0$
 if and only if 
 $$\deg(m)+ 6 = \deg(m_i)+\deg (m_j)$$
 and 
 $$\operatorname{weight}(m)= \operatorname{weight}(m_i) + \operatorname{weight}(m_j).$$
 Since curvature measures of non-zero weight appear only in degrees two and three, this together with \eqref{eq local kinematic SU3} implies \eqref{SU3_local} and
$$K_{\SU(3)}(\Delta_{0,0} ) = K_{\mathrm U(3)}(\Delta_{0,0} )   +a\,  \Phi_3\kfdot \overline{\Phi}_3  
 +b\, \Psi_3\kfdot \overline{\Psi}_3$$
for some constants $a,b$. 
Globalizing in $\C^3$, we get $k_{\SU(3)}(\chi)$ and comparing with  \cite[Theorem 1.6]{bernig_sun}, we find $a= -\frac{2}{15\pi}$. 

By \eqref{eq Euler characteristic}, globalizing 
\[
 K_{\SU(3)}(\Delta_{0,0}+\frac{\lambda}{2\pi}(\Delta_{2,0}+\Delta_{2,1})+\frac{3\lambda^2}{4\pi^2}(\Delta_{4,1}+\Delta_{4,2})+ \frac{15\lambda^3}{8\pi^3}\Delta_{6,3})
\]
in $\S^6$ yields $k_{\G_2}(\chi)$. Using \eqref{eq local Delta0}, \eqref{eq local Delta2} and Proposition~\ref{prop_ker_glob}, we see  $\mu_{4,2}^\lambda\kfdot \mu_{4,2}^\lambda$ appearing here and in 
Proposition \ref{thm_pkf}. Comparing coefficients yields $ b= -\frac{1}{8\pi}$.

 By degree and weight reasons, we obtain \eqref{eq local Phi_2}, \eqref{eq local Psi_2}, and 
 \begin{align*}
  K_{\SU(3)} (\Phi_2) & = 2 \Phi_2\kfdot\Delta_{6,3} +c\, \Phi_3\kfdot\Delta_{5,2} \\
  K_{\SU(3)} (\Psi_2) & = 2 \Psi_2\kfdot\Delta_{6,3} +d\, \Psi_3\kfdot\Delta_{5,2} 
 \end{align*}
 To find the constant $d$, note that by Proposition~\ref{prop_ker_glob}, the images of  
 $\Psi_2$ and  $-\frac{2\sqrt\lambda}{\pi}   (\Delta_{30} -\frac{2}{3}\Delta_{31} + \Phi_3 )$ under 
 $\glob_\lambda\otimes \glob_\lambda \circ K_{\SU(3)}$ coincide. Using \eqref{eq local Delta3}, \eqref{eq local Phi_2}, and Proposition~\ref{prop_ker_glob}, 
 we may compare the coefficient of $\mu^\lambda_{4,1}\kfdot \mu^\lambda_{5,2}$  and find
 $$ 
 - \frac{d\sqrt\lambda }{2}  = -\frac{8\sqrt\lambda}{15\pi}.
 $$
 The same argument shows  $c=0$. 
\end{proof}

\section{Integral geometry of $S^7$ under $\mathrm{Spin}(7)$}

For $\lambda>0$ we  denote  by $\S^7$ the sphere of radius $\lambda^{-1/2}$ centered at the origin of $\O$. For $\lambda=0$ we set $S_0^7=\im \O=T_{e_0} S^7$.
Exploiting  that the restriction of invariant curvature measures on $\S^7$ to $\S^6$ is
essentially injective (Corollary~\ref{Corollary_invariant_restriction}) we give in this section
a thorough account of the integral geometry of $\S^7$ under the action of $\Spin(7)$. 
Our first main result is the second part of Theorem~\ref{Theorem_isomorphism},
which states that $\calV^7_\lambda:=\calV(\S^7)^{\Spin(7)}$ and $\calV^7_0:=\Val(\im \O)^{\mathrm G_2}$ are isomorphic as filtered algebras. 
Our second main result is Theorem~\ref{Theorem_local_kinematic_formulas_G2}, which establishes
 in explicit form the full array of local kinematic formulas for $\mathrm G_2$.

\subsection{Differential forms}

A $3$-form $\sigma\in \largewedge^3 V^*$ on a $7$-dimensional real vector space $V$ is called non-degenerate 
if for each pair $u,v\in V$ of linearly independent vectors there exists $w\in V$ such that $\sigma(u,v,w)\neq 0$. By \cite{salamon_walpuski}*{Remark 3.10},
each non-degenerate form determines a euclidean inner product and an orientation on $V$ (and hence a volume form $\vol_g\in \largewedge^7V^*$) such that 
$$(\iota_u \sigma)\wedge (\iota_v \sigma)\wedge \sigma = 6 g(u,v)  \vol_g \qquad \text{for } u,v\in V.$$
The standard associative $3$-form \eqref{eq associative 3-form} is non-degenerate and induces the standard euclidean inner product and orientation on $\im\O$. 
For every non-degenerate $\sigma$ the group $\{g\in \mathrm{GL}(V)\colon g^*\sigma = \sigma\}$ is isomorphic to $\mathrm G_2$, see, e.g., Theorem~3.2 of \cite{salamon_walpuski}.

A $\G_2$-structure on a $7$-dimensional smooth manifold $M$ is a $3$-form $\sigma\in \Omega^3(M)$ such that $\sigma_p \in \largewedge ^3 T^*_pM$ is non-degenerate for each $p\in M$. 
The  Cayley calibration $\Phi \in \largewedge^4 \O^*$ defines via
$$ \sigma_p=\sqrt{\lambda} \,\iota_p \Phi, \qquad p\in S^7_\lambda,$$
a $\G_2$-structure on $S^7_\lambda$ for $\lambda>0$. For $\lambda=0$ we take $\sigma=\phi$.
The metric and orientation induced by $\sigma$
are the standard metric and orientation on $S^7_\lambda\subset \O$. This may be seen by evaluating $\sigma$ at the representative point $\lambda^{-1/2} e_0$ using \eqref{eq Cayley associative}.
By the same token, $* \sigma = \Phi|_{S^7_\lambda}$ and as a consequence
\begin{equation}\label{eq d sigma}d\sigma = 4 \sqrt{\lambda}*\! \sigma.
\end{equation}

Let $\mathcal{F}_{\G_2}\subset \mathcal{F}_{\mathrm O(7)} \S^7$ be the bundle whose fiber over $p\in \S^7$ consists of orthonormal bases
$u_1,\ldots, u_7$ 
of $T_p\S^7$ satisfying
$$\sigma(u_i,u_j,u_k)= \varepsilon_{ijk},\qquad  1\leq i,j,k\leq 7$$
where the  $\varepsilon$ symbol is skew-symmetric in either three or four indices and is uniquely determined by  
\begin{align*}\phi &= \frac 1 6 \varepsilon_{ijk}\, e^i\wedge e^j \wedge e^k\\
 \psi  &=  \frac{ 1}{ 24} \varepsilon_{ijkl}\, e^i\wedge e^j \wedge e^k\wedge e^l.
\end{align*}  
 Clearly, $\mathcal{F}_{\G_2}$ is a principal bundle with structure group $\G_2\subset \mathrm O(7)$.
 We denote the restrictions of the solder $1$-form $\omega$  and the connection $1$-form $\varphi$  to $\mathcal{F}_{\G_2}$  by the same symbols. 
Let $\SU(3)\subset \G_2$ denote  the stabilizer of the first basis vector and consider in $\mathcal{F}_{\G_2}$ the $1$-form 
$\alpha=\omega_1$,
the $2$-forms 
\begin{align*} 
 \theta_0 &  = \frac 1 2 \varepsilon_{1jk}\, \varphi_{j1} \wedge \varphi_{k1}\\
 \theta_1 & = \varepsilon_{1jk}\,\omega_{j} \wedge \varphi_{k1}\\
 \theta_2 & = \frac 1 2 \varepsilon_{1jk}\, \omega_{j} \wedge \omega_{k}\\
 \theta_s & =\omega_i \wedge \varphi_{i1}
\end{align*}
and the $3$-forms 
\begin{align*}
 \chi_0 & =  \frac 1 6 \left( \varepsilon_{ijk}\,  \varphi_{i1}\wedge \varphi_{j1} \wedge \varphi_{k1}+ \sqrt{-1} \varepsilon_{1ijk}\,  \varphi_{i1}\wedge \varphi_{j1} \wedge \varphi_{k1} \right)\\
 \chi_1 & =  \frac 1 2  \left(\varepsilon_{ijk}\, \theta_{i}\wedge \varphi_{j1} \wedge \varphi_{k1}+ \sqrt{-1} \varepsilon_{1ijk}\, \theta_{i}\wedge \varphi_{j1} \wedge \varphi_{k1} \right)\\
  \chi_2 & =  \frac 1 2  \left(\varepsilon_{ijk}\, \theta_{i}\wedge \theta_{j} \wedge \varphi_{k1}+ \sqrt{-1}  \varepsilon_{1ijk}\, \theta_{i}\wedge \theta_{j} \wedge \varphi_{k1} \right)\\
  \chi_3 & =  \frac 1 6  \left(\varepsilon_{ijk}\, \theta_{i}\wedge \theta_{j} \wedge \theta_{k} + \sqrt{-1}  \varepsilon_{1	ijk}\, \theta_{i}\wedge \theta_{j} \wedge \theta_{k} \right)
\end{align*}
where repeated indices are summed over from $2$ to $7$. By \eqref{eq rg omega} and \eqref{eq rg varphi}, these forms are $\SU(3)$-invariant.
They are also horizontal for the bundle $\mathcal{F}_{\G_2}\to S\S^7$ given by projection to the first basis vector. Hence they descend to the sphere bundle of $\S^7$. 
Again we will make no notational distinction between the forms on the sphere bundle and the corresponding forms on the frame bundle.

\begin{Proposition}\label{prop_basic_forms_span}
 The algebra of $\Spin(7)$-invariant forms on the sphere bundle of $\S^7$ is generated by the forms $\alpha, \theta_0,\theta_1,\theta_2,\theta_s$, and $\chi_i, \overline\chi_i$ for $i=0,\ldots,3$.
\end{Proposition}
\begin{proof}
 Since the stabilizer of  $\xi\in S\S^7$ is $\SU(3)$ we have $T_\xi^* S\S^7 \cong   \R \oplus \C^3\oplus \C^3$  as $\SU(3)$ representations. Hence evaluation 
 at $\xi$ yields an isomorphism
 \begin{equation}\label{eq_iso_invariant_elements}\Omega^k(S\S^7)^{\Spin(7)}\cong \left( \largewedge^k (V^* \oplus V^*) \otimes \C\right)^{\SU(3)} \oplus \left( \largewedge^{k-1} (V^* \oplus V^*) \otimes \C\right)^{\SU(3)}\end{equation}
 where $V=\C^3$ considered as a real vector space. 
The invariant forms introduced above are mapped onto the generators of the algebra $\left( \largewedge (V^* \oplus V^*) \otimes \C\right)^{\SU(3)}$ determined by
Bernig \cite{bernig_sun}*{Lemma~3.3}. Indeed, using \eqref{eq associative 3-form}, its not hard to see that the $\theta_i$ and $\chi_i$
are mapped onto multiples of  the forms $\Theta_i$ and $\Xi_i$ from \cite{bernig_sun}. This concludes the proof.
\end{proof}

\begin{Lemma} \label{lemma_relations_S7}
 $\chi_i\wedge \chi_j=0$ unless $i+j=3$ and the forms $\chi_i\wedge \chi_{3-i}$ for $i=0,\ldots,3$ are all proportional. Moreover
 the forms $\chi_i \wedge \overline\chi_j$ are expressible solely in terms of $\theta_0,\theta_1,\theta_2$, and $\theta_s$.
\end{Lemma}
\begin{proof}
 In the light of \eqref{eq_iso_invariant_elements} this an immediate consequence of the corresponding statement for $\left( \largewedge (V^* \oplus V^*) \otimes \C\right)^{\SU(n)}$ proved in \cite{bernig_sun}.
\end{proof}

\subsection{Curvature measures}
Let $\calC^7_\lambda$ denote the space of invariant curvature measures on $\S^7$. 
We introduce now  a basis of the space  $\calC^7_\lambda\cong \Curv^{\G_2}$. 

\begin{Definition}
 
 For $\max(0,k-3)\leq p\leq k/2\leq 3 $ define $\Theta_{k,p}\in \Curv^{\G_2}$ by
 $$\Theta_{k,p} = \frac{1}{(3+p-k)!(k-2p)!p! \alpha_{6-k}} [ \theta_0^{3+p-k}\theta_1^{k-2p} \theta_2^p, 0]$$
 where $\alpha_k = \vol(S^k)$, and put 
 $$\Phi=\frac{-i}{2\pi^2} [ \chi_0\wedge \chi_3,0].$$
\end{Definition}

Put $\Delta_{k} = \sum_{p} \Theta_{k,p}$ for $k\leq 6$, $\Delta_7(A,U)=\vol_7(A\cap U)$, and 
$$N_2= -\Theta_{2,0} + 4\Theta_{2,1},\quad N_3= 4\Theta_{3,0} - \frac{8}{3}\Theta_{3,1}, \quad N_4= -\Theta_{4,1} + 4\Theta_{4,2}.$$
By Proposition \ref{prop_basic_forms_span} and Lemma \ref{lemma_relations_S7}, the following family spans $\Curv^{\G_2}$:  
\begin{equation}\label{basis_curv_g2}
\Delta_0,\quad
\Delta_1,\quad
\Delta_2,N_2,\quad
\Delta_3, N_3,\Phi,\overline\Phi,\quad
\Delta_4,N_4,\quad
\Delta_5,\quad
\Delta_6,\quad
\Delta_7 
\end{equation}
By Lemma~\ref{restriction_spheres} below, these curvature measures actually form a basis of $\Curv^{\G_2}$. 

If $\iota \colon N\to M$ is a totally geodesic isometric immersion of riemannian manifolds then, by Proposition~\ref{prop filtration transfer}, 
$ \tau\circ \iota^*  = \iota^* \circ \tau$.
 Thus, Corollary~\ref{Corollary_invariant_restriction} implies that the restriction of $\Spin(7)$-invariant curvature measures on $\S^7$ to $\S^6=\S^7\cap \im \O$ is essentially injective. 
 We compute this map now explicitly.

\begin{Lemma}\label{restriction_spheres}
 Let $\iota\colon\S^6 \to\S^7$ be the inclusion. The restriction map $\iota^*\colon \calC^7_\lambda\to \calC^6_\lambda$  is given by
 \begin{align*}
\iota^*(\Delta_7)&=0,\\
\iota^*(\Delta_k)&=\sum_q\Delta_{k,q},\quad k<7\\
\iota^*(N_2)&=N_{2,0}-2\re\Phi_2\\
\iota^*(N_3)&=\frac32\Delta_{3,0}-\Delta_{3,1}-\frac53 N_{3,1}-\frac52\re\Phi_3,\\
\iota^*(N_4)&=-\Delta_{4,1}+4\Delta_{4,2},\\
\iota^*(\re\Phi)&=\frac{3}{8} \Delta_{3,0}-\frac{1}{4} \Delta_{3,1}+\frac{1}{4}N_{3,1} - \frac58 \operatorname{Re} \Phi_{3},\\
\iota^*(\im\Phi) & =\frac{8}{3\pi} \operatorname{Re} \Psi_3.\\
 \end{align*}
\end{Lemma}

  \begin{proof}
 As shown by Alesker in \cite{valmfdsIG}, the pull-back $\iota^*$ of a curvature measure $[\omega,\phi]$ 
 on $Y=\S^7$ to $X=\S^6$   is $[p_*\omega,\pi_*\omega]$, where $\pi_*,p_*$ denote 
 respectively integration along the fibers of $\pi\colon N(\S^6)\to \S^6$ and of the 
 submersion $p\colon E\to S\S^6$, where $E$ is a compact manifold with boundary whose 
 interior is $E^\circ=S\S^7|_X\setminus N(\S^6)$ and $p(x,\cos(t) e_0+\sin(t)\xi)=(x,\xi)$ if $\xi\in S_xX$. 
 Now $f\colon SX \times (0,\pi)\to  E^\circ$, given by $f(u,t)= \cos(t) e_0+\sin(t) u$ is a
 local trivialization for $p|_{E^\circ}$. To compute the fiber integral, we have to compute the pullback under $f$ and integrate with respect to $t$. 
 By invariance, it is enough to do this at the point $u=(\lambda^{-1/2} e_1,e_2)$. Choose some local 
 lift of $f$ to a map $\tilde f$ with values in $\mathcal{F}_{\G_2}(Y)$ such that $\tilde f_1 =f $. We can choose $\tilde f$ so that 
 at the point $(\lambda^{-1/2} e_1,e_2,t)$
 \begin{align*} 
  \tilde f_1 &= \sin(t) e_2 - \cos(t) e_0\\
  \tilde f_2 &= \cos(t) e_2 + \sin(t) e_0\\
  \tilde f_3 &= e_3\\
  \tilde f_4 &= e_4\\
  \tilde f_5 &= -\sin(t) e_7 + \cos(t)e_5\\
  \tilde f_6 &= -\cos(t) e_7 - \sin(t)e_5\\
  \tilde f_7 &= e_6
 \end{align*}
From Lemma~\ref{lemma pullback frame bundle} we obtain
\begin{align*}
\tilde f^* \varphi_{21} & = -dt\\
\tilde f^* \varphi_{31} & = \sin(t) \varphi_{\bar11}\\
\tilde f^* \varphi_{41} & = \sin(t) \varphi_{21}\\
\tilde f^* \varphi_{51} & = \sin(t)( -\sin(t) \varphi_{\bar31} + \cos(t) \varphi_{\bar21})\\
\tilde f^* \varphi_{61} & = \sin(t)( -\cos(t) \varphi_{\bar31} - \sin(t) \varphi_{\bar21})\\
\tilde f^* \varphi_{71} & = \sin(t) \varphi_{31}
\end{align*}
and 
\begin{align*}
\tilde f^* \omega_{2} & = \cos(t)\omega_1\\
\tilde f^* \omega_{3} & = \omega_{\bar1}\\
\tilde f^* \omega_{4} & = \omega_{2}\\
\tilde f^* \omega_{5} & = -\sin(t) \omega_{\bar3}+ \cos(t) \omega_{\bar2}\\
\tilde f^* \omega_{6} & = -\cos(t) \omega_{\bar3}- \sin(t) \omega_{\bar2}\\
\tilde f^* \omega_{7} & = \omega_{3}
\end{align*}
Here the solder and connection forms on the right-hand side refer to the complex basis $e_2,\ldots, e_7$ of $T_{\lambda^{-1/2} e_1}\S^6$. The lemma follows now from direct 
computation. Let us compute for instance the restriction of 
\[
 \im\Phi=\frac{1}{2\pi^2}[\chi_{0,I}\wedge\chi_{3,I}-\chi_{0,R}\wedge\chi_{3,R},0].
\]
Modulo $\omega_1=\alpha$, and modulo terms without a $dt$ factor,  we have
\begin{align*}
 \tilde f^*(\chi_{0,I}\wedge\chi_{3,I})&\equiv\sin^2(t)dt\wedge(-\varphi_{\bar 21}\wedge\varphi_{\bar31}+\varphi_{21}\wedge\varphi_{3 1})\wedge\omega_{\bar 1}\wedge\\
 &\wedge[(-\sin(t)\omega_{\bar3}+\cos(t)\omega_{\bar 2})\wedge\omega_{3} -\omega_2\wedge(-\cos(t)\omega_{\bar3}-\sin(t)\omega_{\bar 2})]\\
 \tilde f^*(\chi_{0,R}\wedge\chi_{3,R})&\equiv \sin^2(t)dt\wedge\omega_{\bar 1}\wedge(-\omega_2\wedge\omega_{3}+\omega_{\bar2}\wedge\omega_{\bar3})\\
 \wedge[\varphi_{21}\wedge&(\cos(t)\varphi_{\bar31}+\sin(t)\varphi_{\bar 21})+(-\sin(t)\varphi_{\bar31}+\cos(t)\varphi_{\bar21})\wedge\varphi_{31})].
\end{align*}
Integrating for $t\in(0,\pi)$ we get
\begin{align*}
 p_*(\chi_{0,I}\wedge\chi_{3,I})&=\frac43\omega_{\bar 1}\wedge(\varphi_{21}\wedge\varphi_{31}-\varphi_{\bar21}\wedge\varphi_{\bar31}) \wedge(\omega_2\wedge\omega_{\bar2}+\omega_3\wedge\omega_{\bar3})\\
 p_*(\chi_{0,R}\wedge\chi_{3,R})&=-\frac43\omega_{\bar 1}\wedge(\varphi_{21}\wedge\varphi_{\bar 21}+\varphi_{31}\wedge\varphi_{\bar31})\wedge(\omega_2\wedge\omega_{3}-\omega_{\bar2}\wedge\omega_{\bar3}).
\end{align*}
It is straightforward to check that $p_*(\chi_{0,I}\wedge\chi_{0,I}-\chi_{0,R}\wedge\chi_{0,R})=\frac43 \beta\wedge\theta_1\wedge\chi_{1,R}$, 
and thus $\iota^*(\im\Phi)=\frac{8}{3\pi}\re\Psi_3$.
The other relations follow similarly.
\end{proof}

\begin{Lemma}\label{lemma_Euler_Verdier_S7} The Euler-Verdier involution on $\Curv^{\G_2}$ is given by 
$$\sigma \Theta_{kp} = (-1)^k \Theta_{kp}$$
 and 
$$ \sigma \Phi  =  -\overline\Phi $$
In particular, $\im\Phi\in \Curv_3^-$.
\end{Lemma}
\begin{proof}
 Since $\varepsilon_{ijk}\neq 0$ if and only if $i+j+k$ is even,  the assignment 
$$ a (e_i) = (-1)^i e_i, \qquad i=1,\ldots, 7,$$
defines a lift of  the fiberwise antipodal map $a\colon SM\to SM$
to the frame bundle $\mathcal{F}_{\G_2} M$ of a manifold with a $\G_2$-structure. Hence
 $$  a^* \omega_i = (-1)^{i} \omega_i\quad\text{and}\quad   a^* \varphi_{ij} = (-1)^{i+j} \varphi_{ij}$$
 for $i,j=1,\ldots, 7$.
As an immediate consequence,
 $$a^*\theta_i = (-1)^{i+1} \theta_i, \qquad i=0,1,2.$$
Since $\varepsilon_{1ijk}\neq 0$ if and only if $i+j+k$ is odd, we also have 
$$a^* \chi_i = (-1)^{i+1} \overline{\chi_i}, \qquad i=0,1,2,3.$$
\end{proof}

\subsection{The algebra of invariant valuations}
A geometric description of the space of translation-invariant and $\G_2$-invariant valuations was obtained by Bernig \cite{bernig_g2}. 
The space $\Val^{\G_2}$ is $10$-dimensional and consists of even valuations. It is spanned by the intrinsic volumes and two new valuations $\nu_3$ and $\nu_4$ 
of degree $3$ and $4$. In terms of Klain functions, 
$$\Kl_{\nu_3}(E) = |\phi|_E|^2\qquad \text{and}\qquad \Kl_{\nu_4}(F) = |\psi|_F|^2,$$
where the norms are understood with respect to the euclidean inner product determined by $\phi$. 
The valuations
\begin{align*}
 \nu_3' & = 5\nu_3- \mu_3\\
  \nu_4' & = 5\nu_4- \mu_4
\end{align*}
belong to non-trivial irreducible $\SO(7)$-subrepresentations of $\Val_k$. 

\begin{Lemma}\label{restriction_val}
Let $i\colon \C^3\cong T_p\S^6\longrightarrow T_pS^7\cong \R^7$ be induced by the inclusion $\S^6\subset\S^7$. 
The corresponding restriction map $i^*\colon \Val^{\G_2}\rightarrow \Val^{\SU(3)}$  is given by
 \begin{align*}
 i^*(\mu_k)&=\sum_q \mu_{k,q}\\
    i^*(\nu_3') & =  \frac 3 2 \mu_{3,0}-\mu_{3,1} -\frac52 \re \phi_3\\
    i^*(\nu_4')& = -\mu_{4,1}+4\mu_{4,2}.
\end{align*}
\end{Lemma}
\begin{proof} In \cite[Lemma 5.1]{bernig_g2},
the restriction $j^*\colon \Val^{\G_2}\rightarrow \Val^{\SU(3)}$ was computed for a linear isometry $j\colon \C^3\rightarrow \R^7$ with
$j^*\sigma=\re\det_\C$. Since $i^*\sigma=\im\det_\C$, we may take $j=i\circ s^{-1}$ where $s$ is any $\C$-linear isometry of $\C^3$ with $\det s=\sqrt{-1}$.  
Since $\mathrm{Kl}_{s^*\phi_3}=-\mathrm{Kl}_{\phi_3}$, comparing with \cite[Lemma 5.1]{bernig_g2} and \cite[(36)]{hig} yields the result. 
\end{proof}

\begin{Corollary}\label{prop glob G2} For $\lambda=0$ 
\begin{equation}\label{eq_glob_intrinsic_volumes}\mu_k= \glob(\Delta_k),\qquad k=0,\ldots,6\end{equation}
and 
$$\nu_l'=\glob(N_l),\qquad l=3,4.$$
\end{Corollary}
\begin{proof}
Since $i^*$ is injective on $\Val_k^{\G_2}$ if $k<7$, and $i^*\circ\glob=\glob\circ \iota^*$, it suffices to check
$$\glob\circ \iota^*( \Delta_k)=\iota^*(\mu_k) \qquad \text{and} \qquad \glob\circ \iota^*( N_l)=\iota^*(\nu_l')$$
which is immediate from Lemmas \ref{restriction_spheres} and \ref{restriction_val}.
\end{proof}

Let $\mu_k^\lambda=[\Delta_k]_\lambda$, $\nu_3^\lambda=[N_3]_\lambda$ and $\nu_4^\lambda=[N_4]_\lambda$
(note that $\nu_3^0=\nu_3'$ and $\nu_4^0=\nu_4'$). By Lemma \ref{lemma_basis} and \cite{bernig_g2}, the valuations  
\[
\mu_0^\lambda,\ldots,\mu_7^\lambda,\nu_3^\lambda,\nu_4^\lambda
\]
constitute a basis of $\calV^7_\lambda$.

Next we determine the algebra structure on $\calV^7_\lambda$.
Put
\[
 \varphi=\frac{2}{\sqrt{\lambda}}\frac{1}{\vol(\S^7)}  k(\chi)(\, \cdot\, ,\S^6),
\]
where $\S^6\subset\S^7$ is a totally geodesic hypersurface. Recall that the formulas \eqref{eq tk_lambda} and \eqref{eq phik} hold.

\begin{Lemma} \label{lemma_action_of_phi_S7} $\varphi\cdot \nu_3^\lambda = -\frac38\nu_4^\lambda$ and $\varphi\cdot \nu_4^\lambda=0$. 
\end{Lemma}
\begin{proof}For $\lambda=0$, the statement was proved in \cite{bernig_g2}.
A dimension count shows that $\glob_0\colon \Curv_4^{\G_2}\to \Val_4^{\G_2}$ is injective. Hence $\varphi\cdot N_3=-\frac38 N_4$ and
$\varphi\cdot N_4=0$ in $S_0^7$. By Proposition \ref{prop_geodesic}, this holds also in $\S^7$ for $\lambda>0$, and the claim follows. 
\end{proof}

The following proves the second part of Theorem~\ref{Theorem_isomorphism}.
\begin{Theorem} The filtered $\C$-algebra $\calV^7_\lambda$ is   isomorphic to 
$$\C[t,u]/(t^2u,u^2+t^6)$$ where the generators $t,u$ have filtration $1$ and  $3$ respectively. 
\end{Theorem}
\begin{proof}For $\lambda=0$ this is proved in \cite{bernig_g2}. 
 By Lemma \ref{lemma_action_of_phi_S7} and \eqref{eq tk_lambda},
 the valuations  $t_\lambda=\frac{2}{\pi}\left.\mu_1\right|_{\S^7}$, and $u_\lambda=\frac{2}{\pi^2}\nu_3^\lambda$
 generate the algebra $\calV^7_\lambda$. By the same token, the  relation $t^2_\lambda u_\lambda=0$ holds. 
Since the relation $u^2_\lambda+t^6_\lambda=0$  holds for $\lambda=0$, Theorem~\ref{graded_iso} and Corollary~\ref{coro_diagram} imply
\[
u^2_\lambda+t^6_\lambda=c\vol
\]
where $c=\pd(u_\lambda,u_\lambda)$. By Lemma \ref{lemma_Euler_Verdier_S7} we have  $\sigma(u_\lambda)=- u_\lambda$ and $\sigma(t_\lambda)=-t_\lambda$.
It follows that $u^2_\lambda+t^6_\lambda$ belongs to the $+1$-eigenspace of $\sigma$. Since $\sigma(\vol)=-\vol$, this shows $c=0$.
Comparing dimensions we see that the map  $t\mapsto t_\lambda$, $u\mapsto u_\lambda$ from
$ \C[t,u]/(t^2u,u^2+t^6)$ to  $\calV^7_\lambda$  is an isomorphism.
\end{proof}
 It is worth mentioning that we have constructed an isomorphism between $\calV^7_\lambda$ and $\Val^{\Spin(7)}$ preserving filtration
 and also Euler-Verdier involution.

We also describe the kernel of the globalization map $\glob_\lambda\colon \Curv^{\G_2} \cong \calC^7_\lambda\to \calV^7_\lambda$.

\begin{Proposition}\label{prop_ker_glob_s7}
 \begin{align*}
  [N_2+\frac{3\lambda}{2\pi}N_4]_\lambda&=0\\
  [ \im \Phi - \frac{4\sqrt{\lambda}}{3\pi} N_4]_\lambda &=0\\
  [N_3-4\re \Phi]_\lambda&=0
 \end{align*}
\end{Proposition}
\begin{proof}
Consider the restriction $r\colon \calV^7_\lambda\to\calV^6_\lambda$. Since $r\circ\glob_\lambda=\glob_\lambda\circ \iota^*$, it follows from Lemma \ref{restriction_spheres}
that $r$ is injective on the subspace spanned by $\mu_0^\lambda,\ldots, \mu_6^\lambda,\nu_3^\lambda,\nu_4^\lambda$.  Hence, $\ker r$ is spanned by $\mu_7^\lambda$. 

Let  $m=N_2+\frac{3\lambda}{2\pi}N_4$. One checks using Lemma \ref{restriction_spheres} and Proposition \ref{prop_ker_glob} that 
the restriction $\iota^*m$ belongs to the kernel of globalization in $\S^6$. Hence, $[m]_\lambda$ is a multiple of $\mu^7_\lambda$. Since $[m]_\lambda(\S^7)=0$ the first equality in the statement is proved. 
The same argument applies also to the other two equalities. 
\end{proof}

In particular, as it was the case in $\Curv^{\SU(3)}$, there is an element of $\Curv^{\G_2}$ which globalizes to zero for every $\lambda$.

\subsection{Global kinematic formulas}

By the fundamental theorem of algebraic integral geometry \eqref{ftaig}, the full array of global kinematic formulas in $\S^7$ under $\Spin(7)$
can be deduced from the  structure of the algebra of invariant valuations. Parallel to the case of $\S^6$ we illustrate this by
providing an explicit expression for the principal kinematic formula. We apply this result to bound the mean intersection number of
associative submanifolds  and arbitrary $4$-dimensional submanifolds of $\S^7$.

Let $k_{\SO(8)}$ denote the kinematic operator in $(\S^7, \SO(8))$. 
\begin{Theorem} The principal kinematic formula in $\S^7$ is given by
$$k_{\Spin(7)}(\chi) = k_{\SO(8)}(\chi) + \frac{1}{256}\nu_3^\lambda\kfdot\nu_4^\lambda.$$
\end{Theorem}
\begin{proof}
Note that Lemma~\ref{lemma_action_of_phi_S7} implies 
\[
 \pd(\nu_i^\lambda,\mu )=0,\qquad i=3,4,\quad \mu\in\calV(\S^7)^{\SO(8)}.
\]
Moreover, since 
\begin{equation}\label{eq pd nu}\nu_3^\lambda\cdot\nu_4^\lambda=512\vol
\end{equation}
holds for $\lambda=0$ by \cite{bernig_g2}, it also holds for any $\lambda>0$ by Theorem \ref{graded_iso} and Corollary \ref{coro_diagram}.
Since $k(\chi)=\pd^{-1}$ as an element of $\calV^7_\lambda\otimes\calV^7_\lambda=\operatorname{Hom}(\calV^{7*}_\lambda,\calV^7_\lambda)$, the claim follows.
\end{proof}

An oriented 3-submanifold $L\subset\S^7$ is called \emph{associative} if $\sigma|_L=\vol_L$, i.e.\ if $L$ is calibrated with respect to $\sigma$. 
Note that coassociative submanifolds of $S^7$, i.e.\ oriented 4-submanifolds $C\subset\S^7$ with $*\sigma|_C=\vol_C$, do not exist. Indeed, $*\sigma|_C=\vol_C$ 
is equivalent to $\sigma|_C=0$. Thus  $d\sigma$ would vanish identically on $C$ which would contradict  \eqref{eq d sigma}.

\begin{Corollary}\label{cor application Spin}
 Let $L\subset S^7$ be a  compact associative submanifold.
 For every  $4$-dimensional  compact submanifold with boundary $M\subset S^7$ the 
 following inequalities hold:
\begin{equation}\label{eq_application_Spin}\frac{15}{128}\vol_4(M)\vol_3(L)\leq  \int_{\Spin(7)} \#( M\cap gL)\, dg<\frac{5}{32}\vol_4(M)\vol_3(L).\end{equation}
 Equality holds on the left if and only if each tangent space of $M$ contains an associative subspace.
\end{Corollary}
\begin{proof}
 For any compact $4$-submanifold with boundary $M$, we have
\begin{align*}
 k_{\Spin(7)}(\chi)(L,M)&=\frac{1}{8}\mu^\lambda_4(M)\mu^\lambda_3(L)+\frac{1}{512}\nu_4^\lambda(M)\nu_3^\lambda(L)\\
 &=\vol_3(L)\left(\frac18 \vol_4(M)+\frac{1}{128}\int_M (5|*\!\sigma|_{T_xM}|^2-1) dx\right).
\end{align*}
Therefore 
\[
 \frac{15}{128}\vol_4(M)\vol_3(L)\leq k_{\Spin(7)}(\chi)(L,M)<\frac{5}{32}\vol_4(M)\vol_3(L)
\]
Equality on the left occurs if and only if  $*\sigma$ vanishes identically on $M$. By Lemma~\ref{lemma assoc} below this happens if and only if 
each tangent space of $M$ contains an associative subspace. The right-hand side inequality is strict since coassociative submanifolds do not exist.
\end{proof}

\begin{Lemma} \label{lemma assoc} Let $V\subset \im \O$ be a $4$-dimensional linear subspace. Then $\psi|_V=0$ if and only if $V$ contains an associative subspace.
\end{Lemma}
 \begin{proof} It is easy to see that $\psi|_V=0$ if $V$ contains an associative subspace. 
 To prove  the converse, note that by \eqref{eq Cayley associative} we have $\Phi|_V=0$, hence $\im(x\times y\times z)\in V^\perp$ for all $x,y,z\in V$. The induced linear map 
  $\largewedge^3V \to V^\perp$ has by dimensional reasons a non-trivial kernel. Since every element in $\largewedge^3V$ is decomposable, 
  there exist $x,y,z\in V$ with $x\wedge y\wedge z \neq 0$ and $\im(x\times y\times z)=0$. Let $U$ be the subspace spanned by these elements.
  Since $|\phi|_U|^2=1$ by the associator equality (cf. \cite{harvey}*{Thm. 7.104}), we have reached the desired conclusion.
 \end{proof}

Note that there are many $4$-dimensional submanifolds $M\subset S^7$ fulfilling the equality conditions in \eqref{eq_application_Spin}. For instance, one may take 
an associative submanifold $L$ in $S^7$ and a smooth family $\{g_t\}\subset \Spin(7)$ such that $M=\bigcup_t g_t(L)\subset S^7$ is a $4$-dimensional 
 compact submanifold with boundary. Since $M$ is foliated by associative submanifolds, every tangent space $T_pM$ contains an associative subspace.

An argument parallel to the one given for \eqref{eq_application_G2} shows that the upper bound in \eqref{eq_application_Spin} is optimal.

\subsection{Local kinematic formulas}
Finally, we obtain the local kinematic formulas under $\Spin(7)$.

 \begin{Lemma} The $\Val^{G_2}$-module structure of $\Curv^{G_2}$ is given by 
  \begin{align*}
  &t\cdot \Delta_j=\frac{2\omega_{j-1}}{\omega_j}\Delta_{j+1}&& \nu_3'\cdot \Delta_0   = 4 \re \Phi\\
  &t\cdot N_2  = \frac{2}{5\pi}(  4\re \Phi  - N_3)&& \nu_3'\cdot \Delta_1   = -\frac{3\pi}{16} N_4\\
  &t\cdot N_3  = -\frac{3}{8} N_4&&\nu_3'\cdot N_3=-30\pi\Delta_6\\
  &t\cdot \Phi  = t\cdot\overline\Phi= -\frac{3}{32} N_4&&\nu_3'\cdot \Phi=\nu_3'\cdot\overline\Phi=-\frac{15\pi}{2}\Delta_6\\
   &&  &\nu_3'\cdot N_4=512 \Delta_7.
 \end{align*}The rest of products of $t,\nu_3'$ with elements of the basis \eqref{basis_curv_g2} vanish.
 \end{Lemma}
\begin{proof}
 The first relation is well-known. The products $t\cdot N_2$ and $\nu_3'\cdot \Delta_0$ can be deduced through restriction to $\C^3$, using the $\Val^{U(n)}$-module structure of $\Curv^{U(n)}$ which was determined in \cite[Section 5]{bfs}. 
  The rest of relations can be deduced from globalization since $\glob$ is injective on $\Curv_k$ for $k\neq 2,3$.
\end{proof}

\begin{Theorem}\label{Theorem_local_kinematic_formulas_G2} The local kinematic formulas for $\G_2$ are given by
 \begin{align*}K_{\G_2}(\Delta_0)&  = K_{\SO(7)}(\Delta_0) + \frac{1}{2^6}  \re \Phi\kfdot N_4 \\
   K_{\G_2}(\Delta_1) & = K_{\SO(7)}(\Delta_1)  -\frac{3\pi}{2^{13}} N_4\kfdot N_4 \\
 K_{\G_2}(\Delta_i)&  = K_{\SO(7)}(\Delta_i)\\
 K_{\G_2}(N_2)&= 2\Delta_7\kfdot N_2+\frac{1}{8}\Delta_6\kfdot (4\re \Phi-N_3)\\
 K_{\G_2}(N_3)&= 2\Delta_7\kfdot N_3-\frac{15\pi}{128} \Delta_6\kfdot N_4\\
 K_{\G_2}(N_4)&= 2\Delta_7\kfdot N_4\\
 K_{\G_2}(\re \Phi)&= 2\Delta_7\kfdot\re \Phi-\frac{15\pi}{512}\Delta_6\kfdot N_4\\
 K_{\G_2}(\im \Phi)&={2\Delta_7\kfdot \im \Phi}.
\end{align*}
where $a\kfdot b=\frac12(a\otimes b+b\otimes a)$.
\end{Theorem}
\begin{proof}
 We determine the semi-local kinematic formulas first. Let us write $K=K_{\mathrm G_2}$, $\bar k= \bar k _{\mathrm G_2}$ for brevity.
 Using the previous lemma we obtain
 \begin{align*}\bar k(\Delta_0) &  
  = k(\chi) \cdot (\chi \otimes \Delta_0) = \bar k_{\SO(7)}(\Delta_0) + \frac{1}{2^9} \nu_3' \otimes N_4+\frac{1}{2^7} \nu_4' \otimes \re \Phi.
 \end{align*}

 Recall that 
 $$(\glob\otimes \operatorname{id}) \circ K = \bar k,$$
 and  $K(\Delta_0)$ is conatined in $W:=\bigoplus_i \Curv_i\otimes\Curv_{7-i}$. By Proposition \ref{prop_ker_glob_s7} (with $\lambda=0$), 
 the kernel of $\glob\otimes \operatorname{id}$ restricted to $W$ is contained in $\bigoplus_{i=2,3}\Curv_i\otimes\Curv_{7-i}$. 
 It follows that $\glob\otimes\operatorname{id}$ is injective on $W\cap\mathrm{Sym}^2\Curv$, where $K(\Delta_0)$ belongs. This proves the first stated relation. 
  For the second and third lines, we use 
 \[
  K(\Delta_i)=(\mu_i\otimes \chi)\cdot K(\Delta_0).
 \]

 Similarly,  
\begin{align*}
 \bar k(N_2)&=\mu_7\otimes N_2 +\frac{1}{16} \mu_6\otimes(4\re \Phi-N_3)\\
 \bar k(N_3)&=\mu_7\otimes N_3-\frac{15\pi}{256}\mu_6\otimes N_4-\frac{15\pi}{256}\nu_4'\otimes \Delta_6+\nu_3'\otimes\Delta_7\\
 \bar k(\im \Phi)&=\mu_7\otimes \im \Phi
\end{align*}
which gives $K(N_2), K(N_3), K(\im \Phi)$. The remaining relations follow for instance from
\[
 K(N_4)=-\frac83(t\otimes \chi)\cdot K(N_3),\qquad K(\re \Phi)=\frac{5\pi}8 (t\otimes \chi)\cdot K(N_2)+\frac14 K(N_3).
\]
\end{proof}

\end{document}